\documentclass[journal]{IEEEtran}
\usepackage{amsmath, amssymb, amsxtra, mathrsfs, amsthm}
\usepackage{graphics, epsfig, here, epstopdf, fancyhdr,graphicx}
\usepackage{bbm}
\usepackage{color}
\usepackage{multirow}
\usepackage[all]{xy}

\newtheorem{theorem}{Theorem}
\newtheorem{lemma}[theorem]{Lemma}
\newtheorem{proposition}[theorem]{Proposition}
\newtheorem{corollary}[theorem]{Corollary}
\newtheorem{assumption}{Assumption}

\newtheorem{definition}{Definition}

\newcommand{\bR}{\mathbb{R}}

\newcommand{\cE}{\mathcal{E}}

\newcommand{\cJ}{\mathcal{J}}
\newcommand{\cL}{\mathcal{L}}

\newcommand{\cV}{\mathcal{V}}

\newcommand{\fA}{\,\forall\,}

\newcommand{\dS}{\displaystyle}

\newcommand{\bq}[1]{{\left[#1\right]}}
\newcommand{\bp}[1]{{\left(#1\right)}}

\newcommand{\vphi}{\varphi}

\renewcommand{\O}{\Omega}

\newcommand{\cG}{\mathcal{G}}

\newcommand{\eP}{\varepsilon}

\newcommand{\pp}{\partial}

\newcommand{\rd}{\mathrm{d}}
\newcommand{\ddx}[1]{{\frac{\rd }{\rd #1}}}
\newcommand{\ppx}[1]{{\frac{\partial }{\partial #1}}}

\usepackage{algorithmic}
\usepackage{epstopdf}

\ifCLASSINFOpdf
\else
\fi
\newcommand{\anatoly}[1]{{\color{black}#1}}

\begin{document}
%
\title{Monotonicity Properties of Physical Network Flows and Application to Robust Optimal Allocation}
%
%
%

\allowdisplaybreaks

\author{Sidhant~Misra,
        Marc~Vuffray,
        and~Anatoly~Zlotnik
        \thanks{S. Misra, M. Vuffray and A. Zlotnik are with Los Alamos National Laboratory, Los Alamos, NM 87545. Email: \{ sidhant $\mid$ vuffray $\mid$ azlotnik \}@lanl.gov.}
 }

\newtheorem{thm}{Theorem}
\newtheorem{lem}[thm]{Lemma}
\newtheorem{prop}[thm]{Proposition}
\newtheorem{cor}[thm]{Corollary}

\newtheorem{df}{Definition}
\newtheorem{nt}{Notation}
\newtheorem{ex}{Example}
\newtheorem{cns}{Construction}
\newtheorem{goal}{Goal}
\newtheorem{rem}{Remark}
\newtheorem{alg}{Algorithm}

\maketitle

\begin{abstract}
We derive conditions for monotonicity properties that characterize general flows of a commodity over a network, where the flow is described by potential and flow dynamics on the edges, as well as potential continuity and Kirchhoff-Neumann mass balance requirements at nodes.  The transported commodity may be injected or withdrawn at any of the network nodes, and its movement throughout the network is controlled by nodal actuators.  For a class of dissipative nonlinear parabolic partial differential equation (PDE) systems on networks, we derive conditions for monotonicity properties in steady-state flow, as well as for propagation of monotone ordering of states with respect to time-varying boundary condition parameters.  In the latter case, initial conditions, as well as time-varying parameters in the coupling conditions at vertices, provide an initial boundary value problem (IBVP).  We prove that ordering properties of the solution to the IBVP are preserved when the initial conditions and the parameters of the time-varying coupling law are appropriately ordered.  Then, we prove that when monotone ordering is not preserved, the first crossing of solutions occurs at a network node.  We consider the implications for robust optimization and optimal control formulations and real-time monitoring of uncertain dynamic flows on networks, and discuss application to subsonic compressible fluid flow with energy dissipation on physical networks. The main result and monitoring policy are demonstrated for gas pipeline test networks and a case study using data corresponding to a real working system.  We propose applications of this general result to the control and monitoring of natural gas transmission networks.
\end{abstract}


\section{Introduction} \label{secintro}

The optimal allocation of commodity flows over networks has been studied from theoretical and computational perspectives since the early work of Ford and Fulkerson, which focused on maximal utilization of capacity and minimization of economic cost in the steady state \cite{ford58b,ford62,ford56,ford58a}.  Subsequent focus has been on maximum profit, or alternatively minimum loss, network flow problems that aim to maximize economic welfare for users of the network by delivering the maximum amount of flow (of a commodity) from sources to terminals.  Prominent algorithms in operations research have particular importance for transportation problems \cite{gass90}, which may involve commodities such as vehicles \cite{ma04,tjandra03phd,como10,10CSADF,13Var_a,13Var_b}, natural gas \cite{osiadacz87,wong68,misra15,luongo91,riosmercado15}, energy \cite{geidl07}, electric power  \cite{Kirchhoff1847,98Bol}, and information \cite{dapice08}.  

Formally, these problems are constructed by extending the standard network flow setting (see e.g. \cite{89AMO,95AMOR} and references therein), with additional physical constraints, introducing nodal potentials, and relating the potential drop along an edge of the network to a function of the flow. Thus, in the case of fluid flows through pipeline in the steady-state regime, the potentials are squared pressures (which are related bijectively to density), and the drop in squared pressure is a bilinear function of the flow and the flow amplitude with a term related to compression \cite{osiadacz87,12BNV,misra15}.

The difficulty of network flow problems is amplified when the flows are unbalanced, i.e., when commodity inflows at origins and outflows at destinations are time-dependent \cite{gottlich05}.  Such situations arise in air traffic flow \cite{ma04}, telecommunications networks \cite{dapice08}, and other flow problems that require dynamic modeling and control design \cite{herty03,como10}.  The number of constraints and decision variables increases by a factor directly related to the temporal complexity of commodity inflows and outflows.  The dynamics are then characterized by systems of ordinary or partial differential equations (ODEs or PDEs), which represent fluid flow or the aggregated motion of discrete particles.  In this context, optimization requires incorporating differential constraints rather than purely algebraic ones, for example in vehicle traffic \cite{gugat05} and gas pipeline flows \cite{steinbach07pde}.

\anatoly{The computational tractability of optimizing dynamic network flows is further challenged by the presence in practice of uncertainty in the volume and timing of the variable commodity inflows and outflows.  In the case of continuous dynamic flows, uncertainty in constant or time-varying functional system parameters, e.g., network inflows and outflows, requires a continuum of constraints to ensure feasibility of the optimization solution.  The area of monotone control systems \cite{angeli03,angeli04} has provided mathemetical mechanisms with which the resulting semi-infinite optimal control problems can be simplified and made tractable for computation.  Developments over the past decade have facilitated stability analysis for systems with monotone order propagation properties \cite{lovisari14}, and enabled robust control in applications including automation of building ventilation systems \cite{meyer13}. Monotonicity properties have been invoked to tractably optimize steady-state fluid flows over networks using variational approaches \cite{dvijotham15,vuffray15cdc}, when the system state can be shown to have a monotone ordering with respect to certain input parameters.  Crucially, this property was shown to enable significant simplification of robust optimization formulations, in particular for distributed flows on large-scale networks.}

\anatoly{In this paper, we begin with a review of the major applications that motivate the robust optimization of uncertain physical flows on networks, as well as of the mathematical developments that lead up to the state of the art.  We then unify results on monotonicity properties for generalized dissipative network flows in the steady state, monotone order propagation properties for transient (unsteady) flows defined by PDEs on graph edges, and preservation of monotone order for spatial discretization schemes of the PDE dynamics.}  Specifically, we show that the steady-states have a monotone ordering with respect to input parameters of practical interest, and in each case derive conditions required for monotonicity properties to hold.  The steady state result is shown to be a special case of the dynamic result, and we also review the connection to a set-theoretic proof \cite{vuffray15cdc}.  Crucially, this property was shown to enable significant simplification of robust optimization formulations, in particular for distributed flows on large-scale networks.  We show that monotonicity renders the adjustable robust maximum profit problem tractable, in the sense that instead of enforcing infinitely many conditions, associated with all possible realizations of the uncertain variables, it is sufficient to only account for two extremal conditions correspondent to every uncertain variable (withdrawal from the network) to greedily maximize/minimize their values. 

We then derive results on the propagation of monotone order properties for systems of nonlinear parabolic PDEs on metric graphs, following our previous study \cite{misra16mtns}.  PDEs with a general nonlinear dissipation term define state evolution on each edge, and balance laws create Kirchhoff-Neumann boundary conditions at the vertices.  We first suppose that initial conditions, together with time-varying parameters that characterize coupling conditions at vertices, provide a well-posed initial boundary value problem (IBVP).  Our main result is a theorem that establishes preservation of monotone ordering properties of the solution to the IBVP when the initial conditions and time-varying coupling law parameters at vertices are appropriately ordered.  Furthermore, we prove that when monotone ordering is not preserved, the first crossing of solutions occurs at a graph vertex.  

In addition, we show that the same conditions for propagation of order hold when the PDE system on a network is discretized in space.  Establishing this property is required to guarantee the same properties in computational implementations of robust optimal network control problems for the examined PDE system, in the presence of uncertainty in nodal commodity withdrawals.  To derive the result, we use the notion of a monotone parameterized control system \cite{zlotnik16ecc}, to which we apply the standard Kamke conditions \cite{kamke32,hirsch05} in order to establish monotonicity with respect to parameter functions.  Lumped-element approximation is used to discretize the dissipative PDEs on network edges as ODE systems, to which existing monotone systems theory can be applied.  Similar to the result for the undiscretized system, we give conditions for the state when the commodity density anywhere in the network can only increase monotonically when any commodity injection is increased.  

\anatoly{The derived monotonicity properties are intended to aid in tractable re-formulation of canonical problems in the mathematical optimization of pipeline transport.  Such problems involve selection of actuator control protocols and/or a subset of injections and withdrawals to optimize an economic or operational cost objective subject to nodal commodity withdrawal limits and bounds on the nodal potentials, with uncertainty in another subset of flows.   When the potentials are monotone functions of the withdrawals, the infinite collection of constraints that enforce the potential limits can be satisfied if the two bounds for the minimum and maximum values of the uncertainty interval are satisfied.   The derived properties are used to compactly formulate tractable robust optimization and optimal control formulations for network flow systems under uncertainty, and to synthesize local feedback policies that maintain monotone state ordering, and can be used to guarantee state feasibility in real-time pipeline operations.}

\anatoly{In order to demonstrate the direct relevance of our results to an application of broad and growing interest, we also present the outcomes of several numerical computations that demonstrate the property that is embodied in our main theorem.  For the motivating application of gas pipeline flow control, we verify the result by confirming the monotone ordering of time-varying physical flow and pressure states computed by simulations of the network dynamics given various monotone ordered time-varying independent (input) parameter functions at the network boundaries.  We consider the examples of a single pipe, a small test network with several compressors, and finally a case study synthesized from a sub-section of capacity planning model of actual pipeline system as well as measurement time-series obtained from its supervisory control and data acquisition (SCADA) system.  }


The manuscript is organized as follows.  \anatoly{In Section \ref{sec:background} we describe the motivating engineering applications, physical modeling, and mathematical developments that inspire this study.} In Section \ref{sec:formulation}, we formulate actuated commodity flows through dissipative transport networks as a class of nonlinear parabolic PDE systems over a collection of domains that form a graph when coupled by Kirchhoff-Neumann boundary conditions.  In Section \ref{sec:discrete_dynamics}, we derive a lumped-element spatial discretization of the continuum dynamics, in which the network is refined and a collection of ODEs that represent nodal density dynamics are obtained.  In Section \ref{sec:result}, we formulate the required assumptions, and state the main results on monotonicity in the theorems on (i) steady state, (ii) monotone order propagation and crossing point conditions for solutions to the PDE system, and (iii) monotone ordering of solutions to the discretized system. We provide formal proofs of the three theorems in Appendices \ref{sec:proof_ss}, \ref{sec:proof_pde}, and \ref{sec:proof_ode}.   Then in Section \ref{sec:discussion} we discuss several properties including uniqueness of steady-state network flow solutions, application to potential difference systems, relationships between the three models and associated theorems, and the application to robust optimal control problems and monitoring (feedback) policies for uncertain dynamic flows on networks. \anatoly{In Section~\ref{sec:example} we  present the outcomes of several computational studies that verify our results for small test networks and a model that represents a working physical system.}  We conclude in Section \ref{sec:conc}.


\anatoly{
\section{Motivation and Background} \label{sec:background}

The theoretical analysis of monotone system properties has largely been driven by optimal transportation problems \cite{como10,vuffray15cdc}.  We first review the applications and how they guide the mathematical setting, and then thoroughly review the relevant mathematical background in the areas of robust optimization, network science, and monotone control systems.


\subsection{Gas Pipeline Systems} \label{subsec:gaspipeflows}

Our main motivation is optimization and control of the complex engineered systems designed for physical transport of natural gas over pipeline networks.  The basic structure of such systems consists of network edges on which a physical flow can be characterized by PDE equations that represent mass and momentum conservation laws.  The edges are connected at nodes where flow balance laws and pressure or density compatibility conditions are additional physical properties.  The flow of compressible gas throughout the network is caused by actuators that can be modeled as nodal or node-connecting. We provide a brief overview of modeling and distinctive physical and mathematical characteristics of these systems, focusing mainly on the conservation laws.  The network modeling is standard, and we define it formally in Section \ref{sec:formulation}.

Gas flow in a pipe is described by the partial differential equation (PDE) system of conservation laws for mass, momentum, and energy in one dimension together with Darcy’s law, which relates mass flow and pressure changes in time and along a pipe segment \cite{wylie78,osiadacz87}.  For large-scale pipeline systems, where most of the pipes are buried underground, the gas temperature is typically the same as that of the ground, except for pipeline sections that are directly connected to the outlet of a compressor.   Therefore in practice the compressibility factor depends on local (in space) pressure and not significantly on temperature, and the analysis of non-isothermal processes may be neglected for the majority of practical cases [9].  It is therefore standard to consider an isothermal process for which the energy conservation law is not required.  The classic transient flow equations for gas flow in a single pipe are 
\begin{subequations} \label{eq:gaspde0}
\begin{align}
    \frac{\partial \rho}{\partial t} + \frac{\partial (\rho u)}{\partial x} & = 0 \label{eq:gaspde0a} \\
    \frac{\partial (\rho u)}{\partial t} + \frac{\partial (p + \rho u^2)}{\partial x} & = - \frac{\lambda}{2D}\rho u |u| - \rho g \frac{d h}{d x} \label{eq:gaspde0b} \\
    p = \rho ZRT & = c_s^2 \rho \label{eq:gaspde0c}
\end{align}
\end{subequations}
Equations \eqref{eq:gaspde0} represent mass conservation, momentum conservation, and the gas equation of state law.  The state variables $u$, $p$, and $\rho$ represent gas velocity, pressure, and density, respectively, and depend on time $t\in[0,T_0]$ and space $x\in(0,L)$, where $T_0$ is a finite time horizon and where $L$ is the length of the pipe, and the variable $h$ gives the elevation of the pipeline.  The dimensionless parameter $\lambda$ is the friction factor that scales the phenomenological Darcy-Weisbach term, which quantifies the momentum loss caused by turbulent friction.  Other parameters are the internal pipe diameter $D$, and the wave (sound) speed  $c_s=\sqrt{ZRT}$ in the gas where $Z$, $R$, and $T$ are the gas compressibility factor, specific gas constant, and absolute temperature, respectively.  Here $dh/dx=\sin(\theta(x))$ where $\theta$ is the angle of the pipe relative to the horizontal, and $g$ is constant acceleration caused by gravity.  In general, the gas compressibility  $Z=Z(p,T)$ depends on pressure and temperature and varies substantially within the physical regime seen in high pressure transmission pipelines.  The term $\partial(\rho u)/\partial t$ in equation \eqref{eq:gaspde0b} represents kinetic energy and the $\partial(\rho u^2)/\partial x$ term represents inertia.  It is standard to apply the transformation of variables by defining the mass flow $\phi=S\rho u$, where $S$ is the cross sectional area of the pipe.  A range of assumptions are made depending on the analysis setting of interest. The common baseline assumptions for gas transmission pipelines are: gas flow is an isothermal process; all pipes are horizontal and have uniform diameter and internal surface roughness; flow is turbulent and has high Reynolds number; and the flow process is adiabatic, i.e. there is no heat exchange with ground. With these assumptions, the coefficients $R$, $T$, $D$, and $\lambda$ can be approximated by constants, and the equations \eqref{eq:gaspde0} can be reduced to
\begin{subequations} \label{eq:gaspde1}
\begin{align} 
    \frac{\partial \rho}{\partial t} + \frac{1}{S}\frac{\partial \phi}{\partial x} & = 0 \\
    \frac{1}{S}\frac{\partial \phi}{\partial t} + \frac{\partial p}{\partial x} & = - \frac{\lambda}{2D}\frac{\phi|\phi|}{S^2\rho} \label{eq:gaspde1:b} \\
    p & = Z(p,T)RT \rho.  \label{eq:gaspde1:c} 
\end{align}
\end{subequations}
There exists a consensus in the literature regarding these initial assumptions. Here we address two subtle modeling points.  There is some controversy about whether it is acceptable to omit the inertial term $\partial \phi/\partial t$ in \eqref{eq:gaspde1:b}, which is often done for mathematical convenience in optimization of transient pipeline flows.  In general, the term can be neglected when the transients are slow, as shown in empirical studies \cite{osiadacz84,gyrya19}.  In order to prove our main result, we must omit the inertial term, and we will show using empirical simulations in Section \ref{sec:example} that our main theorem applies in the regime of slow transients but not when transients are fast.  

A more critical issue is the typical assumption in the majority of gas pipeline optimization studies to assume the ideal gas equation of state, where the gas compressibility $Z$ is a constant.  However for high pressure transmission systems, the gas state becomes highly non-ideal.  The compressibility factor of pipeline quality natural gas, which consists of 95\% methane and 5\% ethane and other longer chain hydrocarbons, can vary between $Z\approx 1$ (at under 2 MPa) to $Z\approx 0.8$ (at over 8 MPa).  This variation occurs throughout time and space and is not known a priori, so that substituting for $p$ in equation \eqref{eq:gaspde1:b} using equation \eqref{eq:gaspde1:c} with constant wave speed significantly changes the solution behavior either for the solution to an IBVP or an optimization problem.

However, we note that equation \eqref{eq:gaspde1:c} can be reformulated when $R$ and $T$ are constant into the form $p=\gamma(\rho)$, where $\gamma$ is a bijective mapping with $\gamma(0)=0$ and derivative satisfying $\gamma'>0$.  Taking the derivative with respect to space results in $\partial_x p=\gamma'(\rho)\cdot \partial_x \rho$. We may then drop the inertial term from equation \eqref{eq:gaspde1:b}, and solve for $\phi$ in the form $\phi=F(t,\rho,\partial_x p)= F(t,\rho,\gamma'(\rho) \partial_x \rho)$, which by the implicit function theorem can in turn be rewritten in the form $\phi=f(t,\rho,\partial_x \rho)$.  In the regime of slow transients and the additional assumptions above, this form of the momentum conservation equation provides a sufficiently accurate representation, and thus we use it as the form of the generalized dissipative relation for gas flow in equation \eqref{eq:in_dissipation_eq} in Section \ref{sec:full_dynamics} and henceforth.

The flow of natural gas through pipelines is propelled by gas compressors, which are of either the centrifugal turbine or reciprocating pump type \cite{mokhatab2012handbook}.  In large gas transmission systems, multiple compressor machines are located at large compressor stations with possibly complex connection topologies \cite{koch15}.  For the purpose of our study, we model such stations as nodal elements that augment gas density between the nodal value and the value at the boundary of a network edge.

The state of the art of modeling, analysis, and computational methods for gas pipeline systems is highly advanced for IBVPs and steady-state optimization.  Recent studies have highlighted the challenge of model predictive optimal control of these complex network systems \cite{zlotnik15cdc,jalving18,zlotnik19cdc}.  The transition of such methods to practice remains an open challenge because of model complexity and parameter uncertainty, which makes tractable optimization problematic.  The main result of our present study provides a powerful theoretical tool, which we demonstrate can addresses this practical challenge.    

\subsection{Additional Physical Network Flow Systems} \label{subsec:physicalflows}

There are additional engineering systems designed for physical transport of commodities over networks, which are structurally similar to gas pipelines, but where the physics of flow on edges are different and the fluid mechanical properties of the flow actuators are more complex.  These include pipeline systems that transport crude petroleum and processed petroleum products, water, or other liquids.

Networks that are used to transport water are analyzed using very similar equations as gas pipeline flows, with the exception that the density is assumed to be constant to account for incompressibility of the fluid \cite{tasseff16}.  Transients for such systems are complex and quickly changing boundary conditions can exhibit water hammer effects \cite{saikia2006simulation}.  Analyses are typically therefore conducted in the steady-state \cite{sarbu1998energetic}, or with the assumption of steady-state flows on a sequence of successive intervals \cite{vrachimis2018real}.  The objective in optimization can be to optimize pumping energy usage \cite{sarbu1998energetic}, as in water distribution pipelines, or to re-allocate flows in open channels to react to flooding or other incidents \cite{tasseff16}.  The modeling of pumping machinery for such systems can become complex as well \cite{georgescu2014estimation}.

For the transport of petroleum and other weakly compressible liquids, an assumption of constant density is also usually made \cite{pharris2008argonne,rizwan2013crude}. Optimization is similarly done by assuming a sequence of intervals that feature steady-state flows \cite{losenokov2017optimization}.  The engineering considerations of pumping systems that actuate flows throughout the system and the resulting effect on physical flows on network edges are critically important. The complex nonlinear relationships between flow rate, rotational drive shaft frequency, head difference, and thermal effects between inlet and outlet of variable frequency drive electric pumps that are predominantly used for petroleum transport are problematic for optimization \cite{rejowski2008novel,12Grishin}, but must be considered to appropriately represent system functions.

Many of the relationships between physical variables and control parameters in commodity transport networks exhibit monotone ordering.  While we focus here on the application to gas pipelines, other network systems provide a rich variety of properties to explore optimization reformulations aided by monotone system theory.


\subsection{Monotonicity and Robust Optimization of Network Flows} \label{subsec:monotonicitybackground}

A significant challenge to computational tractability for practical management of network flows arises through the presence of uncertainty in the volume and timing of the variable commodity inflows and outflows.}  In such settings, it is desirable to extend canonical problems to robust formulations, or more accurately ``adjustable robust optimization'' problems following the terminology commonly accepted in the literature on robust optimization \cite{03BS,ben2009robust,13BNS,13BG, bental98,bertsimas03}.  The robust optimization network flow model considered involves three different types of variables: the uncertain parameters, the non-adjustable variables, and the adjustable variables. The uncertain parameters express information that is not certain, i.e. available for the optimization decision only in the form of allowed range. The non-adjustable variables represent the ``here and now'' decision in the system. Their values should be feasible for any realization of the uncertain parameters from the allowed range. Finally the adjustable variables represent the ``wait and see'' decisions. Their values are adaptable to a particular values of the uncertain parameters.  The adjustable robust optimization is composed of a test of robust feasibility and an optimization procedure. A value of the non-adjustable variables is said to be {\em robust feasible} if for any acceptable configuration of the uncertain variables there exists feasible values of the adjustable variables. Then the optimization procedure consists of finding a robust feasible protocol for setting the non-adjustable variables such that the objective function of overall economic welfare for users of the system is maximized.
Uncertainty can be handled by enforcing a robust feasibility constraint, which is in essence an intersection of infinitely (and possibly uncountably) many constraints, one corresponding to each allowed value of the uncertain parameters. This results in the so-called semi-infinite program \cite{Hettich93}. In the robust optimization literature, the ways to handle these constraints can be classified into three different categories. First, when the constraints and the uncertainty set have special structure, e.g., linear constraints and ellipsoidal uncertainty set, it is possible to use duality theory to represent the infinitely many constraints with one single  dual feasibility constraint \cite{BentalNemirovski98,Nemirovski99,BenTal00}. This category of formulations also includes approximations and relaxations of more complicated uncertainly sets and/or constraints with simpler sets that are amenable to the application of duality theory.  The second category/approach is similar to the so-called ``scenario based'' approach, where  a (possibly random) sampling of the uncertain parameters is performed, and the feasibility constraint corresponding to each sampled parameter is included in the optimization formulation \cite{Bertsimas07, Boyd09}. The quality of the solution thus obtained depends on the number of samples used and also on how the samples were chosen. The third case, which is the approach taken in this manuscript, is when one can analytically or numerically identify the ``extreme-cases'', i.e, find the subset of values of the uncertain parameters that can violate the feasibility constraints. When this subset is finite, or has a finite representation, the robust feasibility constraint again reduces to a finite number of standard constraints. Examples where this strategy is used are scarce. (See \cite[pp. 388]{Boyd09} for a discussion on the topic.)

In the case of continuous dynamic flows, uncertainty in constant or time-varying functional system parameters, e.g., network inflows and outflows, requires a continuum of constraints to ensure feasibility of the optimization solution.  As in the steady-state case, the challenge becomes to similarly show that feasibility for a finite number of appropriate scenarios will guarantee feasibility for an entire such uncountable ensemble of constraints.  A recent approach to control uncertain network flows with time-dependent dynamics sidesteps the need for global optimization over a possibly non-convex landscape by examining stability and robustness of distributed routing solutions \cite{como13a,como13b}.  The methodology in these studies was enabled by demonstrating that the dynamics in question were monotone control systems \cite{angeli03,como10,lovisari14,sootla2018operator}.  Such so-called cooperative systems, which possess a monotone order propagation property with respect to certain input variables, were investigated for ordinary differential equation systems \cite{kamke32,hirsch85,smith88,hirsch05}.   The recent discovery of numerous applications has renewed interest in such systems, for example to vehicle routing under uncertainty \cite{como10}, analysis of chemical reaction networks \cite{deleenheer04}, as well as power systems and turbulent jet flows \cite{budivsic12}.  The notion of monotone control systems \cite{angeli03,angeli04} has also facilitated stability analysis for systems with monotone order propagation properties \cite{lovisari14}, and enabled robust control in applications including automation of building ventilation systems \cite{meyer13}.  Several results on the propagation of order properties for stochastic systems exist as well \cite{sootla14b}.

Previous studies on monotone dynamical systems have largely focused on monotone order propagation properties of ordinary differential equations (ODEs) \cite{hirsch05}, and applications involving representations of fluid flow or the aggregated motion of discrete particles were examined with ODE models \cite{hirsch05,como13a}.  However, control and optimization approaches for systems represented by PDE dynamics could benefit significantly from monotone systems concepts, in particular control of fluid flows on networks \cite{steinbach07pde,zlotnik15cdc} and quantum graphs \cite{arioli16}.  Studies that have used monotonicity properties to optimize steady-state fluid flows over networks using variational approaches \cite{dvijotham15,vuffray15cdc} demonstrate that the steady-states have a monotone ordering with respect to certain input parameters.  Crucially, this property was shown to enable significant simplification of robust optimization formulations, in particular for distributed flows on large-scale networks.  

The need to develop robust optimal control formulations for emerging applications involving uncertain dynamic flows on networks motivates investigation of monotone order propagation properties for PDE systems as well as the associated discretization schemes.  The approximation of a diffusive PDE operator by an ODE system and derivation of order propagation properties using the established ODE theory has been suggested for basic reaction-diffusion problems \cite{deleenheer04, enciso06}.  Otherwise, monotone operators have been examined primarily in the context of existence and approximations of solutions to nonlinear PDE systems \cite{quaas08,briani12,showalter13}.  In the following exposition, we formalize and unify the notations used to study monotone system properties for physical network flows.


\section{Parabolic PDE Systems on Metric Graphs} \label{sec:formulation}
We consider a metric graph $\Gamma=\left(\cV,\cE,\lambda\right)$ where $\cV$ is the set of vertices and $\cE\subset \cV \times \cV$ is the set of directed  edges $(i,j)\in\cE$ that connect the vertices $i,j\in\cV$.  Here $\lambda:\cE\to\bR_+$ is a metric on the edges, where $\bR_+$ denotes the non-negative real numbers. Let the incoming and outgoing neighborhoods of $j\in \cV$ be denoted by $\partial_{+}j$ and $\partial_{-}j$, respectively.  These sets are defined as
\begin{align}
\partial_{+}j&=\left\{ i\in \cV\mid(i,j)\in \cE\right\} \\
\partial_{-}j&=\left\{ k\in \cV\mid(j,k)\in \cE\right\}.
\end{align}
Every edge $(i,j)\in \cE$ is associated with a spatial dimension on the interval $I_{ij}=[0,L_{ij}]$, where $L_{ij}=\lambda(i,j)>0$ is interpreted as the edge length defined by the metric $\lambda$.  We let $V=|\cV|$ and $E=|\cE|$ denote the number of vertices and of edges, respectively.

\subsection{Full System Dynamics} \label{sec:full_dynamics}
The state of the network system is characterized within each edge $(i,j)\in\cE$  by space-time dependent variables corresponding to flow $\phi_{ij}:[0,T]\times I_{ij}\rightarrow\mathbb{R}$ and non-negative density $\rho_{ij}:[0,T]\times I_{ij}\rightarrow\mathbb{R}_{+}$. In addition, every vertex $i\in \cV$ is associated with a time-dependent {internal} nodal density $\rho_{i}:[0,T]\rightarrow\mathbb{R}_{+}$ and is subject to
a time-dependent flow injection $q_{i}:[0,T]\rightarrow\mathbb{R}$.

We suppose that the density and flow dynamics on the edge  $(i,j)\in\cE$ evolve according to the generalized dissipative relations,
\begin{align}
\dS \pp_t\rho_{ij}(t,x_{ij})+\pp_x\phi_{ij}(t,x_{ij}) & =  0, \label{eq:in_continuity} \\
\phi_{ij}(t,x_{ij})+f_{ij}(t,\rho_{ij}(t,x_{ij}), \partial_{x}\rho_{ij}(t,x_{ij})) & =0,  \label{eq:in_dissipation_eq}
\end{align}  
which are called respectively the continuity and momentum dissipation equations. The functions $f_{ij}(t,u,v):[0,T]\times \mathbb{R}_{+} \times \mathbb{R} \rightarrow \mathbb{R}$ are called dissipation functions and we assume that they are increasing in their last argument.



Next, we establish nodal relations that characterize the boundary conditions for the flow dynamics \eqref{eq:in_continuity}-\eqref{eq:in_dissipation_eq} on each edge of the graph.  For this purpose, in order to simplify notation we define the edge boundary variables
\begin{align}
\underline{\rho}_{ij}(t)\triangleq\rho_{ij}(t,0), \quad \overline{\rho}_{ij}(t)\triangleq\rho_{ij}(t,L_{ij}), \label{eq:end_p_def} \\
\underline{\phi}_{ij}(t)\triangleq\phi_{ij}(t,0), \quad \overline{\phi}_{ij}(t)\triangleq\phi_{ij}(t,L_{ij}). \label{eq:end_q_def}
\end{align}
At each vertex $i\in V$ the flow and density values at the endpoints of adjoining edges must satisfy certain compatibility conditions.  First, a Kirchhoff-Neumann property of flow conservation is ensured through nodal continuity equations
\begin{align}
q_j(t)+\sum_{i\in\partial_{+}j}\overline{\phi}_{ij}- \sum_{k\in\partial_{-}j}\underline{\phi}_{jk}=0, \quad \fA j\in\cV. \label{eq:in_nodal_continuity} 
\end{align}
In addition, we include compatibility conditions that relate nodal densities to boundary conditions on edges.  
For each edge $(i,j)\in\cE$, the corresponding nodal conditions are
\begin{align}
\underline{\rho}_{ij}(t) = \underline{\alpha}_{ij} (t,\rho_{i}(t)), \quad
\overline{\rho}_{ij}(t)  = \overline{\alpha}_{ij}(t,\rho_{j}(t)), \label{eq:in_pressure_comp} 
\end{align}
where the compatibility functions $\underline{\alpha}_{ij}(t,\rho)$ and $ \overline{\alpha}_{ij}(t,\rho)$ are monotonically increasing functions in $\rho$ for all $t\in[0,T]$ and $\rho>0$.
The functions $\rho_{i}$ are auxiliary variables that denote internal nodal density values.  The above compatibility conditions are visualized in Figure \ref{fig:compatibility}.

We suppose that instantaneous state of the system at time $t=0$ is specified by initial density and flow profiles
\begin{align}
\!\!\! \rho_{ij}(0,x)=\rho_{ij}^{0}(x), \,\, \phi_{ij}(0,x)=\phi_{ij}^{0}(x), \quad \fA (i,j)\in\cE. \label{eq:in_initial_condition}
\end{align}

\begin{figure}[t]
\centering
\includegraphics[width=.95\linewidth]{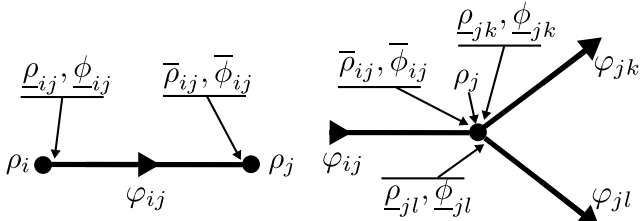} \vspace{-2ex} \caption{Nodal densities $\rho_j$ and boundary variables $\underline{\rho}_{ij}$, $\underline{\phi}_{ij}$, $\overline{\rho}_{ij}$, and $\overline{\phi}_{ij}$, and compatibility functions $\underline{\alpha}_{ij}$ and $\overline{\alpha}_{ij}$
for an edge (left) and a joint (right).
} \label{fig:compatibility}
\vspace{-2ex}
\end{figure}
With the above setting, we obtain an initial boundary value problem (IBVP) where the initial conditions are given by \eqref{eq:in_initial_condition}.

\subsection{Steady State} \label{sec:steady_state}
The steady state of the system described in Section~\ref{sec:full_dynamics}, assuming it exists, can be obtained by removing the time dependence and setting the terms with time derivatives in the PDEs to zero. Setting $\dS \pp_t\rho_{ij}(t,x_{ij})$ to zero in \eqref{eq:in_continuity} we get $\pp_x\phi_{ij}(t,x_{ij}) = 0$, which implies that the flow $\phi_{ij}(t,x_{ij})$ is constant over an edge. We can define the corresponding simplified steady state system states as density $\rho_{ij}(x_{ij})$ and flow $\phi_{ij}$ (where the dependence on $x_{ij}$ is no longer required and thus removed). 
The steady state is described by the following system of equations:
\begin{subequations} \label{eq:steady_state}
\begin{align}
    &\phi_{ij} = f_{ij}(\rho_{ij}(x_{ij}), \partial_{x}\rho_{ij}(x_{ij}) ) \quad \fA (i,j) \in\cE \label{eq:ss_dissipation} \\
    &q_j+\sum_{i\in\partial_{+}j}{\phi}_{ij}- \sum_{k\in\partial_{-}j}{\phi}_{jk}=0, \quad \fA j\in\cV  \label{eq:ss_flow_conservation} \\
    &\underline{\rho}_{ij} = \underline{\alpha}_{ij} (\rho_{i}), \quad
\overline{\rho}_{ij}  = \overline{\alpha}_{ij}(\rho_{j})    \quad \fA (i,j) \in\cE,\label{eq:ss_compressor}
\end{align}
\end{subequations}
where $\underline{\alpha}_{ij}(.)$ and $\overline{\alpha}_{ij}(.)$ are monotonically increasing functions for all $(i,j) \in \cE$. 

\section{Discretized System Dynamics} \label{sec:discrete_dynamics}
For applications involving the system described in Section~\ref{sec:full_dynamics} that require computational simulation and/or optimal control, it is necessary to obtain a finite representation of the PDEs by performing discretization in space and time. In this section, we present the system of coupled ODEs obtained by using a lumped element approximation on the PDEs.

We use a lumped element approximation \cite{heydweiller1977dynamic,grundel13a} to characterize edge dynamics \eqref{eq:in_continuity} and \eqref{eq:in_dissipation_eq}, with nodal conditions \eqref{eq:in_nodal_continuity} and \eqref{eq:in_pressure_comp} and subject to injection profiles $q_{i}(t)$,  which approximately defines the state on the network in terms of nodal densities $\rho_j(t)$. Our approach is to add enough nodes to the network so that density and flow are nearly uniform on any given segment. 
In particular, we obtain dynamic equations where the state is represented by the vector of nodal densities $\rho=(\rho_1,\ldots,\rho_V)$.  We begin with the following definition.

\begin{df}[Spatial Graph Refinement] \label{def:graphref} The refinement $\hat{\cG}_{\eP}=(\hat{\cV}_{\eP},\hat{\cE}_{\eP},\hat{\lambda}_{\eP})$ of a weighted oriented graph $\cG=(\cV,\cE,\lambda)$ is made by adding nodes to $\cV$ to sub-divide edges of $\cE$ where the length $\hat{L}_{ij}\in\hat{\cL}_{\eP}$ of a new edge $(i,j)\in\hat{\cE}_{\eP}$ satisfies
\begin{align}
\frac{\eP L_{\mu(ij)}}{\eP+L_{\mu(ij)}}<\hat{L}_{ij}<\eP,
\end{align}
where $\mu:\hat{\cE}\to\cE$ is an injective map of refined edges to the parent edges in $\cE$.
\end{df}

\begin{rem} Spatial graph refinement preserves the structure of the network represented by the graph, and can finely discretize the coupled one-dimensional domains on which the network dynamics \eqref{eq:in_continuity}-\eqref{eq:in_dissipation_eq} with \eqref{eq:in_nodal_continuity}-\eqref{eq:in_pressure_comp} evolve.  For $\eP\ll\min_{{i,j}\in\cE} L_{ij}$, the lengths in $\hat{\cL}_{\eP}$ are nearly uniform and very close to $\eP$.
\end{rem}

\begin{rem} We assume that $\eP$ is small enough so that the relative difference of density and flux at the start and end of each new edge $(i,j)\in\hat{\cE}_{\eP}$ is small.  Specifically,
\begin{align} \label{eq:pres_rel}
2\frac{\overline{\rho}_{ij}(t)-\underline{\rho}_{ij}(t)}{ \overline{\rho}_{ij}(t)+\underline{\rho}_{ij}(t)} \ll 1, \,\, 2\frac{\overline{\phi}_{ij}(t)-\underline{\phi}_{ij}(t)}{ \overline{\phi}_{ij}(t)+\underline{\phi}_{ij}(t)} \ll 1, \,\, \fA t
\end{align}
for the transient regime of interest.  In other words, $\eP$ is sufficiently small so that the relative density difference between neighboring nodes is very small at all times.
\end{rem}

Figure \ref{fig:netconstit} presents an illustration of an edge junction (left picture) and an edge segment (right picture) of a spatial graph refinement $\hat{\cG}_{\eP}$ with  $V_{\eP}=|\cV_{\eP}|$ nodes and $E_{\eP}=|\cE_{\eP}|$ edges of approximate length $\eP$. The variable $q_j$ is an injection into the network at node $j$, and $\overline{\O}_{ij}$ and $\underline{\O}_{jk}$ are  sub-elements corresponding to halves of incoming and outgoing edges $(i,j)$ and $(j,k)$ in $\hat{\cE}_{\eP}$. The flow at the midpoint of an edge is denoted by $\vphi_{ij}=\phi_{ij}(t,\hat{L}_{ij}/2)$. The densities $\underline{\rho}_{ij}$ and $\overline{\rho}_{ij}$ at the ends of the edge $(i,j)\in\hat{\cE}_{\eP}$ are related to the nodal densities $\rho_i$ and $\rho_j$ by Equation \eqref{eq:in_pressure_comp}, as described in Section \ref{sec:formulation}.


\begin{figure}[t]
\centering{
\includegraphics[width=.95\linewidth]{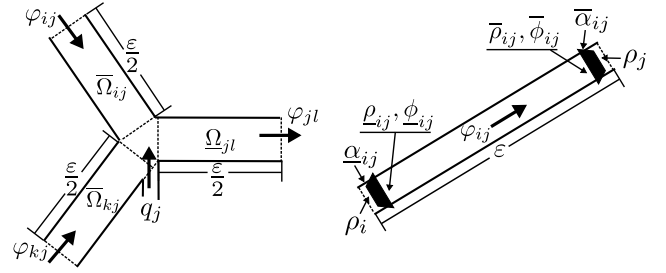} \caption{Lumped elements for discretization of equations \eqref{eq:in_continuity} at a node (left) and \eqref{eq:in_dissipation_eq} over an edge (right). The actuators located between the boundaries of an edge and the adjacent nodes are represented by trapezoids.} \label{fig:netconstit} }
\vspace{-2ex}
\end{figure}

We approximate the rate of change of mass within a nodal element in the refined graph by summing the integrals of mass flux gradient on each adjoining edge segment.  That is,
\begin{align}
&\sum_{i\in\partial_+j}\int_{\overline{\O}_{ij}}\partial_x\phi_{ij}(t,x_{ij}) +\sum_{k\in\partial_-j}\int_{\underline{\O}_{jk}}\partial_x\phi_{jk}(t,x_{jk}) \label{eq:fx_constit1}  \\
&\qquad=  \sum_{i\in\partial_+j}(\overline{\phi}_{ij}-\vphi_{ij}) +\sum_{k\in\partial_-j}(\vphi_{jk}-\underline{\phi}_{jk})  \label{eq:fx_constit2}\\
&\qquad= \sum_{k\in\partial_-j}\vphi_{jk} - \sum_{i\in\partial_+j}\vphi_{ij} -q_j,
\label{eq:fx_constit3}
\end{align}
where the last step is due to the nodal balance condition \eqref{eq:in_nodal_continuity}. Next, applying mass conservation \eqref{eq:in_continuity} to \eqref{eq:fx_constit1} results in
\begin{align}
&\sum_{i\in\partial_+j}\int_{\overline{\O}_{ij}}\partial_x\phi_{ij}(t,x_{ij}) +\sum_{k\in\partial_-j}\int_{\underline{\O}_{ij}}\partial_x\phi_{jk}(t,x_{jk}) \label{eq:pt_constit1}  \\
&\,\, =  -\!\!\sum_{i\in\partial_+j}\int_{\overline{\O}_{ij}}\!\!\partial_t \rho_{ij}(t,x_{ij}) - \!\!\sum_{k\in\partial_-j}\int_{\underline{\O}_{ij}}\!\!\partial_t \rho_{jk}(t,x_{jk})   \label{eq:pt_constit2}\\
&\,\,  \approx -\sum_{i\in\partial_+j}\frac{\eP}{2}\dot{\overline{\rho}}_{ij} -\sum_{k\in\partial_-j}\frac{\eP}{2}\dot{\underline{\rho}}_{jk}
\label{eq:pt_constit3} \\
&\,\, = \!\! -\sum_{i\in\partial_+j}\frac{\eP}{2}\ddx{t}(\overline{\alpha}_{ij}(t,\rho_{j})) -\sum_{k\in\partial_-j}\frac{\eP}{2}\ddx{t}(\underline{\alpha}_{jk}(t,\rho_{j}))
\label{eq:pt_constit4} \\
&\,\, =  -\frac{\eP}{2}\!\!\sum_{i\in\partial_+j}\!\left( \frac{\partial}{\partial t} \overline\alpha_{ij}(t,\rho_j) + \frac{\partial}{\partial \rho}\overline\alpha_{ij}(t,\rho_j) \dot\rho_j   \right) \nonumber \\
&\,\,\quad -\frac{\eP}{2}\!\!\sum_{k\in\partial_-j}\!\left( \frac{\partial}{\partial t} \underline\alpha_{jk}(t,\rho_j) + \frac{\partial}{\partial \rho}\underline\alpha_{jk}(t,\rho_j) \dot\rho_j  \right)  \\
&\,\, = -\frac{\eP}{2}\left( \frac{\partial}{\partial t}\alpha_j(t,\rho_j) + \frac{\partial}{\partial \rho}\alpha_j(t,\rho_j)\dot\rho_j  \right), \label{eq:pt_constit6}
\end{align}
where $\alpha_j(t,\rho_j)$ denotes aggregated actuation at node $j\in\hat{\cV}_{\eP}$,
\begin{align}
\alpha_j(t,\rho_j)= \sum_{i\in\partial_+j} \overline{\alpha}_{ij}(t,\rho_j) + \sum_{k\in\partial_-j} \underline{\alpha}_{jk}(t,\rho_j).
\label{eq:alpha_def}
\end{align}
The approximation in \eqref{eq:pt_constit3} is made by assuming sufficient network refinement \eqref{eq:pres_rel}, and the nodal density relations \eqref{eq:in_pressure_comp} are substituted into \eqref{eq:pt_constit3} to obtain \eqref{eq:pt_constit4}.  We have established equality of \eqref{eq:fx_constit3} and \eqref{eq:pt_constit6}, so solving for $\rho_j$ yields the discretized nodal mass conservation dynamics
\begin{align}
\!\!\!\!\!\! \dot{\rho}_j = & \dS \frac{2}{\eP \frac{\partial}{\partial \rho}\alpha_j(t,\rho_j)}\bq{\sum_{i\in\partial_+j}\!\!\vphi_{ij} \!\!  -  \!\!\!\dS\sum_{k\in\partial_-j}\!\!\vphi_{jk} \!+ q_j} \nonumber  \\ & \,\,\,\, - \frac{\frac{\partial}{\partial t}\alpha_j(t,\rho_j)}{\frac{\partial}{\partial \rho}\alpha_j(t,\rho_j)}, \, \fA j\in\hat{\cV}_{\eP}.  \label{eq:disc_mass_balance}
\end{align}

We now approximate the dissipation equation \eqref{eq:in_dissipation_eq} by evaluating the spatial gradient with a finite difference
\begin{align}
\partial_{x}\rho_{ij}(t,x_{ij}) \approx \frac{1}{\eP}(\overline{\rho}_{ij}-\underline{\rho}_{ij}) = \frac{1}{\eP}(\overline{\alpha}_{ij}(t,\rho_j)-\underline{\alpha}_{ij}(t,\rho_i)), \label{eq:px_grad}
\end{align}
accounting for endpoint actuators as shown at right in Figure~\ref{fig:netconstit}.  Applying \eqref{eq:px_grad} to approximate \eqref{eq:in_dissipation_eq} at nodes yields
\begin{align}
 {\vphi}_{ij} &= \dS -f_{\mu(ij)}\bp{\! t,\overline{\alpha}_{ij}(t,\rho_{j}), \frac{1}{\eP}(\overline{\alpha}_{ij}(t,\rho_j)-\underline{\alpha}_{ij}(t,\rho_i)) \!}, \nonumber \\ & \qquad \qquad \qquad \qquad \qquad \qquad \qquad \qquad  \fA i\in\partial_+j,  \label{eq:disc_diss_eq1} \\
 {\vphi}_{jk} &= \dS -f_{\mu(jk)}\bp{\! t,\underline{\alpha}_{jk}(t,\rho_{j}), \frac{1}{\eP}(\overline{\alpha}_{jk}(t,\rho_k)-\underline{\alpha}_{jk}(t,\rho_j)) \!}, \nonumber \\ & \qquad \qquad \qquad \qquad \qquad \qquad \qquad \qquad \fA  k\in\partial_-j,  \label{eq:disc_diss_eq2}
\end{align}
where $(i,j)$ and $(j,k)$ are used for incoming and outgoing edges at a node, respectively.  Substituting \eqref{eq:disc_diss_eq1}-\eqref{eq:disc_diss_eq2} into \eqref{eq:disc_mass_balance} produces the following purely nodal dynamics:

\begin{align}
 & \dot\rho_j = \frac{2}{\eP} \left(\frac{\partial}{\partial \rho}\alpha_j(t,\rho_j) \right)^{-1} \times \nonumber \\
  & \left[ \sum_{k\in\partial_-j} f_{\mu(jk)}\bp{\! t,\underline{\alpha}_{jk}(t,\rho_{j}), \frac{1}{\eP}(\overline{\alpha}_{jk}(t,\rho_k)-\underline{\alpha}_{jk}(t,\rho_j)) \!}\right. \nonumber \\
  & \,\,\, -\sum_{i\in\partial_+j} f_{\mu(ij)}\bp{\!t,\overline{\alpha}_{ij}(t,\rho_{j}),\frac{1}{\eP}(\overline{\alpha}_{ij}(t,\rho_j)-\underline{\alpha}_{ij}(t,\rho_i))} \nonumber \\
 & \quad \Bigl.+ q_j \Biggr]
 - \frac{\frac{\partial}{\partial t}\alpha_j(t,\rho_j)}{\frac{\partial}{\partial \rho}\alpha_j(t,\rho_j)} , \quad \fA j\in\hat{\cV}_{\eP}.  \label{disceq3}
\end{align}

\begin{rem}{Regularity Assumptions.} \label{rem:mol} First, we note that the ODE system \eqref{disceq3} is defined on the nodes $\hat{\cV}_\eP$ of the $\eP$-refined graph $\cG_\eP$.  We assume that this discretization scheme for the PDE system defined by \eqref{eq:in_continuity}-\eqref{eq:in_dissipation_eq} with \eqref{eq:in_nodal_continuity}-\eqref{eq:in_pressure_comp} is convergent and stable in the sense of a method of lines (MOL) solution.  That is, we suppose that the distance between solutions to \eqref{disceq3} and the classical solution to the PDE system defined by equations \eqref{eq:in_continuity}-\eqref{eq:in_pressure_comp} at locations corresponding to refined network nodes will converge point-wise to zero as $\eP\to 0$.
\end{rem}


\section{Monotone Order Properties} \label{sec:result}
Our main results establish certain monotone order preserving properties for the non-linear parabolic PDE systems described in Section~\ref{sec:formulation}. We derive such results for the system dynamics described in Section~\ref{sec:full_dynamics}, its steady-state in Section~\ref{sec:steady_state}, as well as the system of ODEs that describe the discretized dynamics in Section~\ref{sec:discrete_dynamics}. We state these results below, and then provide the proofs in Sections \ref{sec:proof_ss}, \ref{sec:proof_pde}, and \ref{sec:proof_ode}, respectively.

\subsection{Steady-State Result} \label{subsec:steadystatethm}

We begin by stating the result for the simplest setting -- the steady-state system. The only assumption we require for the result in Theorem~\ref{thm:ss} below is that the ODE imposed by the dissipation equation on edges in \eqref{eq:ss_dissipation} admits a unique solution. For examples of systems where this property exists, refer to Section~\ref{sec:discussion}. 

\begin{assumption}
\label{thm:ss_assumption}
Consider the dissipation function $f_{ij}$ in the dissipation equation \eqref{eq:ss_dissipation}. Let $\rho_{ij}(x_c) = \rho_0$ for some $x_c \in [0,L_{ij}]$ and $\phi_{ij} = \phi_0$ be given. Then for all admissible $x_c\in [0,L_{ij}]$, $\rho_c\geq0$ and $\phi_{0} \in \mathbb{R}$, the ODE in the dissipation equation along with the above initial conditions admit a unique solution, i.e., there exists a unique trajectory $\rho(x)$ for $x \in [0,L_{ij}]$ with $\rho(x_c) = \rho_0$ such that 
\begin{align}
f_{ij}(\rho(x), \pp_x \rho(x)) + \phi_{0}  = 0.
\end{align}
\end{assumption}
\vspace{1ex}

\begin{theorem}
\label{thm:ss}
Suppose that the dissipation function $f_{ij}(u,v)$ in \eqref{eq:ss_dissipation} is strictly increasing in the second argument $v$ for all $(i,j) \in \cE$ and that Assumption~\ref{thm:ss_assumption} is satisfied. Consider two sets of flow injections $q_i^{(1)}$ and $q_i^{(2)}$ associated with densities $\rho_i^{(1)}$ and $\rho_i^{(2)}$ respectively, satisfying \eqref{eq:steady_state}. Let $\mathcal{S} \subseteq \cV$ be an arbitrary subset of $\cV$ such that for all $i \in \mathcal{S}$ we have $q_i^{(1)} \geq q_i^{(2)}$ and for all $i \in \cV \setminus \mathcal{S}$ we have $\rho_i^{(1)} \geq \rho_i^{(2)}$. Then all the nodal densities in the system satisfy $\rho_i^{(1)} \geq \rho_i^{(2)}$ for all $i \in \cV$ and all edge densities satisfy $\rho_{ij}^{(1)}(x_{ij}) \geq \rho_{ij}^{(2)}(x_{ij})$ for all $x_{ij} \in [0,L_{ij}]$ and all $(i,j) \in \cE$.
\end{theorem}
 
\subsection{Full PDE system result} \label{subsec:fullpdethm}

Next, we state the result for the PDE form of the full system dynamics given in Section~\ref{sec:full_dynamics}. We first state a set of regularity conditions that we impose on these dynamics. 
\begin{assumption} \label{assumption}
We make the following assumptions on initial value problem \eqref{eq:in_continuity}-\eqref{eq:in_initial_condition} that describes the coupled network flow dynamics with initial conditions.
\begin{itemize}
\item[(i)] \emph{Well-posedness and regularity of initial conditions:} The initial densities and flows are twice continuously differentiable, i.e. $\rho_{ij}^{0},\,\phi_{ij}^{0}\in C^2([0,L_{ij}])$ for all $(i,j)\in\cE$. Moreover the coupling constraints \eqref{eq:in_nodal_continuity} and \eqref{eq:in_pressure_comp} hold at $t=0$.
\item[(ii)] \emph{Continuity of inputs and control:} The compatibility functions are twice continuously differentiable, i.e. $\underline{\alpha}_{ij},\,\overline{\alpha}_{ij}\in C_+^2([0,T]\times\bR_+)$ for all $(i,j)\in\cE$,
as well as the nodal parameter functions $q_i\in C^2([0,T])$ for all $i\in\cV$.
\item[(iii)] \emph{Well-posedness of coupled network dynamics:} The initial value problem consisting of the coupled network flow dynamics with the initial conditions in \eqref{eq:in_continuity}-\eqref{eq:in_initial_condition}, along with given
compatibility functions $\overline{\alpha}_{ij}$ and $\underline{\alpha}_{ij}$, admits a unique classical solution that is twice continuously differentiable.
\item[(iv)] \emph{Stability under small perturbations:} Let $\rho_{ij}(t,x_{ij}), \phi_{ij}(t,x_{ij})$ for $(i,j) \in \cE$ be the unique classical solution to \eqref{eq:in_continuity}-\eqref{eq:in_initial_condition}. Let
$\rho_{ij,\epsilon}(t,x_{ij})$ and $\phi_{ij,\epsilon}(t,x_{ij})$ for all $(i,j) \in \cE$ be a solution to the perturbed system
\begin{align}
\!\!\!\!\!\!\!\!\!   \dS \pp_t\rho_{ij,\epsilon}(t,x_{ij})\!+\!\pp_x\phi_{ij,\epsilon}(t,x_{ij})  - \epsilon & =  0,  \label{perturbed1} \\
\!\!\!\!\!\!\!\!\!  \phi_{ij,\epsilon}(t,x_{ij})\!+\!f_{ij}(t,\rho_{ij,\epsilon}(t,x_{ij}), \partial_{x}\rho_{ij,\epsilon}(t,x_{ij}))& =0,  \label{perturbed2}
\end{align}
with the perturbed initial conditions
\begin{align}
\!\!\! \rho_{ij,\epsilon}(0,x)=\rho_{ij}^{0}(x) + \epsilon, \quad \phi_{ij,\epsilon}(0,x)=\phi_{ij}^{0}(x)  \label{eq:in_initial_condition_pert}
\end{align}
for all $(i,j)\in\cE$.
Then as $\epsilon \rightarrow 0$, the perturbed solution converges point-wise to the original solution, i.e., for all $(i,j) \in \cE$, $\ x_{ij} \in I_{ij}$  and  $t \in [0,T]$, we have
\begin{align}
\lim_{\epsilon \rightarrow 0} \rho_{ij,\epsilon}(t,x_{ij}) = \rho_{ij}(t,x_{ij}).
\end{align}
\end{itemize}
\end{assumption}

The monotone order propagation property for the full dynamics described in Section \ref{sec:full_dynamics} is stated below. 
\begin{theorem}
\label{thm:pde}
Suppose the initial value problem described in \eqref{eq:in_continuity}-\eqref{eq:in_initial_condition} satisfies Assumption \ref{assumption}. Also suppose that the dissipation function
$f_{ij}(t,u,v)$ is strictly increasing in the third argument $v$ for all $(i,j) \in \cE$. Let $\rho_{ij}^{(1)}(0,x_{ij})$ and  $\rho_{ij}^{(2)}(0,x_{ij})$ be two initial conditions that satisfy $\rho_{ij}^{(1)}(0,x_{ij}) \geq \rho_{ij}^{(2)}(0,x_{ij})$ for  all $(i,j) \in \cE$, $x_{ij} \in I_{ij}$. Let $\mathcal{S} \subseteq \cV$ be an arbitrary subset of $\cV$. Let $t_0 \in [0,T]$ and suppose that for all $i \in \mathcal{S}$ we have that $q_i^{(1)}(t) \geq q_i^{(2)}(t)$
for all $t \in [0,t_0]$ and for all $i \in \cV \setminus \mathcal{S}$ we have that $\rho_i^{(1)}(t) \geq \rho_i^{(2)}(t)$ for all $t \in [0,t_0]$. Then  the densities in the system satisfy $\rho_{ij}^{(1)}(t,x_{ij}) \geq \rho_{ij}^{(2)}(t,x_{ij})$
for all $(i,j) \in \cE$, $x_{ij} \in I_{ij}$ and $t \in [0,t_0]$.
\end{theorem}

Assuming that the dynamics converge to the steady state solution when the input is time-independent, then the monotone order property in Theorem~\ref{thm:pde} can be used to prove Theorem~\ref{thm:ss} by using the continuity property in Assumption~\ref{assumption} (ii). However, the result in Theorem~\ref{thm:ss} holds with no additional assumptions other than the necessary uniqueness of solutions in Assumption~\ref{thm:ss_assumption}, afforded by a very different first principle based proof technique, and is worth stating separately. 

\subsection{Spatially discretized ODE system result}  \label{subsec:discretizedthm}
Our next result asserts that the monotone order propagation property is preserved when the full PDE dynamics in Section~\ref{sec:full_dynamics} are discretized to the system of ODEs as described in Section~\ref{sec:discrete_dynamics}.
\begin{theorem}
\label{thm:ode}
Suppose that the dissipation function
$f_{ij}(t,u,v)$ is strictly increasing in the third argument $v$ for all $(i,j) \in \cE$.
Let $\rho_j^{(1)}(0)$ and $\rho_j^{(2)}(0)$ be two initial conditions such that $\rho_j^{(1)}(0) \geq \rho_j^{(2)}(0)$ for all $j \in \cV$. Let $\mathcal{S} \subseteq \cV$ be an arbitrary subset of $\cV$. Suppose that for all $i \in \mathcal{S}$ we have that $q_i^{(1)}(t) \geq q_i^{(2)}(t)$
for all $t \in [0,T]$ and for all $i \in \cV \setminus \mathcal{S}$ we have that $\rho_i^{(1)}(t) \geq \rho_i^{(2)}(t)$ for all $t \in [0,T]$. Then the evolution of nodal density variables for the system, according to Eq.~\eqref{disceq3}, satisfies $\rho_i^{(1)}(t) \geq \rho_i^{(2)}(t)$ for all $t \in [0,T]$ and for all $i \in \cV$.
\end{theorem}

Although Theorem~\ref{thm:pde} shows that the physical system described by the full dynamics preserves monotone ordering, it is not always the case that a discretization scheme inherits this property \cite{lipnikov2007monotone}. Theorem~\ref{thm:ode} shows that a simple Euler discretization scheme successfully inherits this property. As described later in Section~\ref{sec:discussion}, a discretization scheme that preserves monotone ordering is crucial to obtain tractable formulations in applications involving robust optimal control and security monitoring.

\section{Discussions and Applications} \label{sec:discussion}

The theorems provide a powerful set of conceptual tools that can be applied to establish additional properties of systems that feature physical flows on networks.  These monotone ordering property results can furthermore be applied to the development of algorithms for control and optimization of such systems, which they can greatly simplify by using formulations that are specific to how flows throughout the network are monitored and controlled.

\subsection{Uniqueness of Solutions for Steady-State Systems} \label{subsec:uniqueness}

Theorem~\ref{thm:ss}, which guarantees that the system is monotonic in the densities with respect to a change in the input parameters, enables us to prove that a solution for the network flow problem is necessarily unique.
Consider two solutions $(q^{(1)},\phi^{(1)},\rho^{(1)})$ and $(q^{(2)},\phi^{(2)},\rho^{(2)})$ of the steady-state problem given by Eq.~\eqref{eq:steady_state} where input parameters have been prescribed at a subset of nodes $\mathcal{S}\subseteq \cV$ and the nodal densities have been prescribed at the complementary subset of nodes $\cV\setminus\mathcal{S}$. The monotonicity Theorem~\ref{thm:ss} immediately implies that the densities of both solutions are equal everywhere as $q^{(1)}_i = q^{(2)}_i$ for $i\in \mathcal{S}$ and $\rho^{(1)}_j = \rho^{(2)}_j$ for $j\in \cV\setminus\mathcal{S}$. Since the densities are equal everywhere, it also implies that the flows and the input parameters at remaining nodes $\cV\setminus\mathcal{S}$ are equal as they are uniquely determined from the densities using Eq.~\eqref{eq:steady_state}.


\subsection{Potential Difference for Steady-State Systems} \label{subsec:potentialdiff}
In the steady-state regime, when the dissipation function is the composition of an increasing function $g$ and the derivative of an increasing function $h(\rho)$, 
\begin{align}
f_{ij}(\rho_{ij}(x_{ij}), \partial_x \rho_{ij}(x_{ij})) &= g\left(\frac{d}{dx}h(\rho_{ij}(x_{ij}))\right)\nonumber\\
&= g\left(h'(\rho_{ij}(x_{ij}))\partial_x \rho_{ij}(x_{ij})\right),
\end{align}
the system is integrable and can be expressed exclusively using nodal quantities.
The ODE equation for the flow now reads 
\begin{align}
g^{-1}(\phi_{ij}) = -h'(\rho_{ij}(x_{ij}))\partial_x \rho_{ij}(x_{ij}),
\end{align}
where $g^{-1}$ is the inverse function of $g$. This equation can be integrated exactly over $[0,L_{ij}]$ and gives
\begin{align}
g^{-1}(\phi_{ij}) L_{ij} = h(\rho_{ij}(0))-h(\rho_{ij}(L_{ij})).
\end{align}
The two terms $h(\underline \rho_{ij}) = \underline \psi_{ij}$ and $h(\overline \rho_{ij}) = \overline \psi_{ij}$ can be interpreted as a ``potential", and the flow is then induced by difference in this potential, and moreover is from higher to lower potential as the function $g$ is increasing. 
\begin{align}
\overline \psi_{ij} - \underline \psi_{ij} = -\frac{1}{L_{ij}} g^{-1}_{ij}(\phi_{ij}). \label{eq:pd_ss}
\end{align}
Equation \eqref{eq:pd_ss} closely resembles a resistive electric circuit with a non-linear dissipation term given by $g^{-1}(.)$.




\subsection{Application to Robust Optimal Control} \label{subsec:robustcontrol}

The monotonicity properties established above have several important implications for formulating robust optimal control problems for parabolic PDE systems on networks.  Here we consider a robust control formulation where the nodal parameter functions $q_i(t)$ are prescribed within a compact subset of twice continuously differential functions $C^2[0,T]$, but are uncertain.  Control formulations have been developed to address problems related to the transportation of commodities over networks, in particular the flow of compressible fluids such as natural gas in large scale pipeline systems \cite{herty10,zlotnik15cdc}, where the physical flows are described by systems of the form of equations \eqref{eq:in_continuity}-\eqref{eq:in_pressure_comp}.  However, the addition of uncertainty to the parameters in such problems, specifically in the consumption of the transported commodity by consumers throughout the network, requires the formulation and solution of the robust optimal control problem.  A major motivation for the development of the monotonicity theory presented here is its use in formulating computationally tractable and scalable algorithms for such problems. Consider the following deterministic optimal control problem:

\begin{subequations}\label{eq:ocp0}
\begin{align}
\!\!\! \min \,\,\, & \cJ(\rho, \phi, \alpha) = \int_0^T \cL(t,\underline{\phi}(t),\overline{\phi}(t),\underline{\alpha}(t),\overline{\alpha}(t)) dt, \label{eq:det_obj} \\
\!\!\!\mbox{s.t.} \,\,\, & \dS \pp_t\rho_{ij}(t,x_{ij})+\pp_x\phi_{ij}(t,x_{ij})  =  0  \label{eq:opc0_b}\\
\!\!\! &\phi_{ij}(t,x_{ij})+f_{ij}(t,\rho_{ij}(t,x_{ij}), \partial_{x}\rho_{ij}(t,x_{ij}))  =0, \label{eq:opc0_c}\\
\!\!\! &\underline{\rho}_{ij}(t) \!=\! \underline{\alpha}_{ij}\rho_{i}(t), \,\, \overline{\rho}_{ij}(t)  \!=\! \overline{\alpha}_{ij}\rho_{j}(t), \, \fA (i,j)\in\cE \label{eq:opc0_d} \\
\!\!\! &q_j(t)+\sum_{i\in\partial_{+}j}\overline{\phi}_{ij}- \sum_{k\in\partial_{-}j}\underline{\phi}_{jk}=0, \quad \fA j\in\cV \label{eq:opc0_e} \\
\!\!\! & \rho_{min} \leq \rho_{ij}(t,x_{ij}) \leq \rho_{max}, \,\, \fA (i,j)\in\cE, \label{eq:det_ineq} \\
\!\!\! & e_\rho(\rho_{ij}(0,x_{ij}),\rho_{ij}(T,x_{ij}))) = 0 \,\, \fA (i,j)\in\cE, \label{eq:terminal_a} \\
\!\!\! & e_\phi(\phi_{ij}(0,x_{ij}),\phi_{ij}(T,x_{ij}))) = 0 \,\, \fA (i,j)\in\cE. \label{eq:terminal_b} 
\end{align}
\end{subequations}
The above formulation is a minimal abstraction of the practical problem of pipeline optimal control \cite{zlotnik2017economic,zlotnik19cdc}, in which the PDE dynamic constraints \eqref{eq:in_continuity}-\eqref{eq:in_dissipation_eq}, nodal flow balance constraints \eqref{eq:in_nodal_continuity}, and nodal density (actuation) compatibility constraints \eqref{eq:in_pressure_comp} are represented by \eqref{eq:opc0_b}-\eqref{eq:opc0_c}, \eqref{eq:opc0_d}, and \eqref{eq:opc0_e}, respectively.  Here the density compatibility functions are linear with factors $\underline{\alpha}_{ij}$ and $\overline{\alpha}_{ij}$, which are not otherwise constrained.  In practice, complex constraints on gas compressor actuators must be enforced.  The constraints \eqref{eq:det_ineq} are invoked to enforce operational requirements to maintain system pressurization, and the equations \eqref{eq:terminal_a} and \eqref{eq:terminal_b} represent criteria on the initial and terminal states.  We leave these ambiguous here, because characterization of these conditions to guarantee well-posed optimization problems in function space remains an open problem. In defining the objective function, we write $\underline{\phi}(t)=\{\underline{\phi}_{ij}(t)\}_{(ij)\in E} \in \bR^{E}$ and $\overline{\phi}(t)=\{\overline{\phi}_{ij}(t)\}_{(ij)\in E} \in \bR^{E}$.  Similarly, we write $\underline{\alpha}(t)=\{\underline{\alpha}_{ij}(t)\}_{(ij)\in E} \in \bR^{E}$ and $\overline{\alpha}(t)=\{\overline{\alpha}_{ij}(t)\}_{(ij)\in E} \in \bR^{E}$, which form the collection of control functions.
Here we formulate a version of problem \eqref{eq:ocp0} in which the solution is feasible given instances of the nodal parameter functions $q_i(t)$ within some known bounds, i.e.,
\begin{align}
q^{(1)}_j(t) \geq q_j(t) \geq q^{(2)}_j(t),  \quad \fA j\in\cV \text{ and } t\in[0,T].  \label{box}
\end{align}
A solution to the resulting problem, which we call robust to uncertain variation, is extremely challenging because of the semi-infinite set of constraints (\eqref{eq:opc0_b}-\eqref{eq:det_ineq} must be satisfied for all values of $q_i(t)$ in \eqref{box}).  Using Theorem~\ref{thm:pde} (or one of Theorem~\ref{thm:ss} or Theorem~\ref{thm:ode} as the setting may require) however, we can compose a dramatically simplified reformulation of the robust control problem \eqref{eq:ocp0} with \eqref{box} as well as a ``monitoring'' mechanism that we will describe in a subsequent subsection.

\subsection{Simplified Representation of Robust Optimal Control} \label{subsec:robustreformulation}

As a consequence of Theorem~\ref{thm:pde}, we can obtain a reformulation of the semi-infinite constrained robust control problem with interval uncertainty as specified in the equation \eqref{box} by enforcing feasibility only for the extreme scenarios.  In particular, we seek an optimal solution that is simultaneously feasible for two scenarios - one in which the maximum load (corresponding to the minimum injections $q^{(2)}$) occurs while system pressures must be maintained above the minimum limits; and the other in which the minimum load occurs (corresponding to the maximum injections $q^{(1)}$) while the pressures are maintained below the maximum limits.  As long as the optimal control solution satisfies the constraints for the two extremal cases of nodal parameter functions $q_i(t)$ for $j\in\cV$, feasibility will also be guaranteed for all nodal parameter functions that are bounded by the extreme scenarios. 

\newpage

\noindent We state the entire formulation below for completeness:
\begin{subequations} \label{eq:ocp1}
\begin{align}
\!\!\! \min \,\,\, & \cJ(\rho, \phi, \alpha) \! =\!\! \int_0^T \!\!\!\!\cL(t,\underline{\phi}(t),\overline{\phi}(t),\underline{\alpha}(t),\overline{\alpha}(t)) dt, \label{eq:rob_obj} \\
\!\!\!\mbox{s.t.} \,\,\, & \dS \pp_t\rho_{ij}(t,x_{ij})+\pp_x\phi_{ij}(t,x_{ij})  =  0  \\
\!\!\! &\phi_{ij}\!(t,x_{ij}) \!+\! f_{ij}(t,\rho_{ij}(t,x_{ij}), \partial_{x}\rho_{ij}(t,x_{ij})) \!=\! 0, \\
\!\!\! &\underline{\rho}_{ij}\!(t) \!=\! \underline{\alpha}_{ij}\rho_{i}\!(t), \,\, \overline{\rho}_{ij}\!(t)  \!=\! \overline{\alpha}_{ij}\rho_{j}\!(t), \, \fA \! (\!i\!,\!j\!)\!\in\!\cE \\
\!\!\! &\hat{q}_j(t)+\sum_{i\in\partial_{+}j}\overline{\phi}_{ij}- \sum_{k\in\partial_{-}j}\underline{\phi}_{jk}=0, \quad \fA j\in\cV \\
\!\!\! & e_\rho(\rho_{ij}(0,x_{ij}),\rho_{ij}(T,x_{ij}))) = 0 \,\, \fA (i,j)\in\cE, \\
\!\!\! & e_\phi(\phi_{ij}(0,x_{ij}),\phi_{ij}(T,x_{ij}))) = 0 \,\, \fA (i,j)\in\cE, \\
\!\!\! & \dS \pp_t\rho_{ij}^{(1)}(t,x_{ij})+\pp_x\phi_{ij}^{(1)}(t,x_{ij})  =  0  \\
\!\!\! &\phi_{ij}^{(1)}\!(t,x_{ij}) \!+\! f_{ij}(t,\rho_{ij}^{(1)}(t,x_{ij}), \partial_{x}\rho_{ij}^{(1)}(t,x_{ij})) \!=\! 0, \\
\!\!\! &\underline{\rho}_{ij}^{(1)}\!(t) \!=\! \underline{\alpha}_{ij}\rho_{i}^{(1)}\!(t), \,\, \overline{\rho}_{ij}^{(1)}\!(t)  \!=\! \overline{\alpha}_{ij}\rho_{j}^{(1)}\!(t), \, \fA \! (\!i\!,\!j\!)\!\in\!\cE \\
\!\!\! &q_j^{(1)}(t)+\sum_{i\in\partial_{+}j}\overline{\phi}_{ij}^{(1)}- \sum_{k\in\partial_{-}j}\underline{\phi}_{jk}^{(1)}=0, \quad \fA j\in\cV \\
\!\!\! & e_\rho(\rho^{(1)}_{ij}(0,x_{ij}),\rho^{(1)}_{ij}(T,x_{ij}))) = 0 \,\, \fA (i,j)\in\cE, \\
\!\!\! & e_\phi(\phi^{(1)}_{ij}(0,x_{ij}),\phi^{(1)}_{ij}(T,x_{ij}))) = 0 \,\, \fA (i,j)\in\cE, \\
\!\!\! & \dS \pp_t\rho_{ij}^{(2)}(t,x_{ij})+\pp_x\phi_{ij}^{(2)}(t,x_{ij})  =  0  \\
\!\!\! &\phi_{ij}^{(2)}\!(t,x_{ij}) \!+\! f_{ij}(t,\rho_{ij}^{(2)}(t,x_{ij}), \partial_{x}\rho_{ij}^{(2)}(t,x_{ij})) \!=\! 0, \\
\!\!\! &\underline{\rho}_{ij}^{(2)}\!(t) \!=\! \underline{\alpha}_{ij}\rho_{i}^{(2)}\!(t), \,\, \overline{\rho}_{ij}^{(2)}\!(t)  \!=\! \overline{\alpha}_{ij}\rho_{j}^{(2)}\!(t), \, \fA \! (\!i\!,\!j\!)\!\in\!\cE \\
\!\!\! &q_j^{(2)}(t)+\sum_{i\in\partial_{+}j}\overline{\phi}_{ij}^{(2)}- \sum_{k\in\partial_{-}j}\underline{\phi}_{jk}^{(2)}=0, \quad \fA j\in\cV \\
\!\!\! & e_\rho(\rho^{(2)}_{ij}(0,x_{ij}),\rho^{(2)}_{ij}(T,x_{ij}))) = 0 \,\, \fA (i,j)\in\cE, \\
\!\!\! & e_\phi(\phi^{(2)}_{ij}(0,x_{ij}),\phi^{(2)}_{ij}(T,x_{ij}))) = 0 \,\, \fA (i,j)\in\cE, \\
\!\!\! & \rho_{min} \leq \rho_{ij}^{(1)}(t,x_{ij}) \leq \rho_{max}, \,\, \fA (i,j)\in\cE \label{eq:rob_ineq1} \\
\!\!\! & \rho_{min} \leq \rho_{ij}^{(2)}(t,x_{ij}) \leq \rho_{max}, \,\, \fA (i,j)\in\cE. \label{eq:rob_ineq2}
\end{align}
\end{subequations}

In the above formulation, the objective function $\cJ$ is defined in terms of flows and actuation factors $\underline{\phi}$, $\overline{\phi}$, $\underline{\alpha}$, and $\overline{\alpha}$ for the primary problem, which are defined by the optimal solution given the nominal injection profiles $\hat{q}_i(t)$ that are bounded by the extremal envelopes.  The optimization is subject to dynamic constraints that describe the effects of the controls $\underline{\alpha}$ and $\overline{\alpha}$, as in the deterministic problem \eqref{eq:ocp0}, for a nominal case $\hat{q}$ of the injections. In addition, we apply PDE dynamic constraints, nodal flow balance constraints, nodal density  compatibility constraints, and terminal constraints for application of the same optimized controls, $\underline{\alpha}$ and $\overline{\alpha}$, in the high injection case $q^{(1)}$ for the variables $\rho^{(1)}$ and $\phi^{(1)}$, and similarly in the low injection case $q^{(2)}$ for the variables $\rho^{(2)}$ and $\phi^{(2)}$.  This enforces feasibility of the $\rho_{ij}(t,x_{ij})$ only with respect to the extreme scenarios corresponding to the lower and upper envelopes $q_i^{(1)}(t)$ and $q_i^{(2)}(t)$ of the uncertain nodal parameters $q_i(t)$.  By Theorem~\ref{thm:pde}, as long as $q_i^{(1)}(t) \geq \hat{q}_i(t) \geq q_i^{(2)}(t)$ holds for all $i\in\mathcal V$, then the corresponding densities $\rho_{ij}(t,x_{ij})$ must also satisfy $\rho_{ij}^{(1)}(t,x_{ij}) \geq \rho_{ij}(t,x_{ij}) \geq \rho_{ij}^{(2)}(t,x_{ij})$ for all $t\in[0,T]$, and therefore the constraints in \eqref{eq:rob_ineq1} and \eqref{eq:rob_ineq2} are satisfied as well.    Furthermore, if the objective function is also monotone with respect to the nodal input parameters, one can obtain a simplified representation of min-max robust optimal control, along the lines of \cite{vuffray15cdc}.

In this example we have considered a formulation that is robust with respect to the so-called interval constraints in Eq.~\eqref{box} as nodal parameter functions are constrained independently from each others. However, the motoniticy properties can be applied to reduce arbitrary uncertainty sets in a robust formulation. In this case motoniticy properties guarantees that realizations of nodal parameter functions $q_i(t)$ can be removed from the uncertainty set if there exists there exists $q_i^{(1)}(t)$ and $q_i^{(2)}(t)$ in the uncertainty set for which
$q^{(1)}_j(t) \geq q_j(t) \geq q^{(2)}_j(t)$,  $\fA j\in\cV \text{ and } t\in[0,T]$.
\subsection{Real-Time Nodal Monitoring Policy} \label{subsec:nodalpolicy}
In \eqref{eq:ocp1} we provide a tractable formulation for the robust control problem for which the interval  envelopes for the uncertainty in the nodal injection parameters $q_i(t)$ are known a priori.  We now suppose that there exists a feasible solution to the simplified representation for the robust optimal control problem \eqref{eq:ocp1}, and that we have determined the optimal control vectors $\underline{\alpha}(t)$ and $\overline{\alpha}(t)$ that maintain system feasibility under all possible instances of the uncertain injection profiles.  A simple monitoring mechanism is proposed here, which can be used to respond to and correct real-time deviations outside of the predicted uncertainty envelope.  If, for instance, an error in uncertainty quantification causes $q_i(t)$ for some $i \in \cV$ to deviate outside of the envelope $[q_i^{(2)}(t),q_i^{(1)}(t)]$ used to specify the box constraints \eqref{box}, then application of the control solution $\underline{\alpha}(t)$ and $\overline{\alpha}(t)$ to problem \eqref{eq:ocp1} no longer guarantees an acceptable solution because the assumptions of Theorem~\ref{thm:pde} no longer hold.  One way to compensate for such variation is to fix the nodal input parameters $q_i(t)$ to the upper or lower bounds, $q_i^{(1)}(t)$ or $q_i^{(2)}(t)$, as appropriate, at the time when the crossing would occur. However, this may be too conservative for enforcing the density inequality constraints \eqref{eq:det_ineq}, because $q_i(t)$ may be outside of the feasibility envelope without violation of the density constraints, which are critical in practice. However, Theorem~\ref{thm:pde} still applies, and this facilitates a much less conservative Nodal Monitoring Policy (NMP) to reactively maintain system densities within feasible values.

\vspace{0.1in}
\noindent \emph{\bf Nodal Monitoring Policy (NMP):} Let $\overline{\cal{S}} \subset \cV$ be the subset of nodes where the upper bound $q_i(t) \leq q_i^{(1)}(t)$ on nodal parameters is violated, and suppose that $\underline{\cal{S}} \subset \cV$ is the subset of nodes where the lower bound $q_i(t) \geq q_i^{(2)}(t)$ is violated. Let $\rho_i^{(1)}(t)$ and $\rho_i^{(2)}(t)$ for $i\in\cV$ be the collections of nodal density solutions that correspond to fixing the nodal input parameters at $q_i^{(1)}(t)$ and $q_i^{(2)}(t)$, respectively. The policy is to monitor the real-time density profiles $\rho_i(t)$ for $i \in \underline{\cal{S}} \cup \overline{\cal{S}}$. If no crossing points are encountered between the real-time solution $\rho_i(t)$ and the upper and lower density profile solutions $\rho_i^{(1)}(t)$ and $\rho_i^{(2)}(t)$, respectively, then the system is safe with respect to the density limits. Alternatively, suppose we encounter a crossing point at time $t_c$ at node $i \in \overline{\cal{S}}$. Then the policy will be to reset the nodal injection parameter at $i$ to $q_i^{(1)}(t)$ while leaving the remaining nodal parameters unchanged.  This simple action is sufficient to guarantee that the system-wide density profiles remains within $\rho_{\min}$ and $\rho_{\max}$. This is contained in the following corollary.

 \begin{corollary}[Sufficiency of Nodal Monitoring Policy] Suppose that the NMP is implemented as described above. It follows that $\rho^{(2)}(t,x_{ij}) \leq \rho(t,x_{ij}) \leq \rho^{(1)}(t, x_{ij})$ for all $t \in [0,T]$ and $(i,j) \in \cE$.
\end{corollary}
\begin{proof} The proof is a direct application of Theorem~\ref{thm:pde}, because by construction of the policy, the assumptions in the theorem are satisfied for all $t \in [0,T]$.
\end{proof}

\begin{figure}[!t]
\centering{
\includegraphics[width=\linewidth]{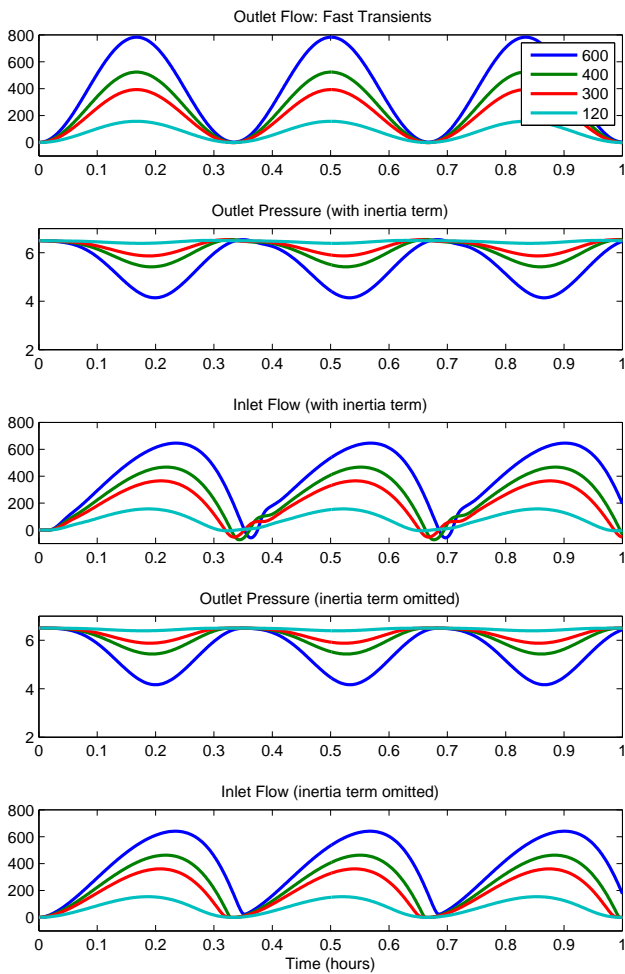} \caption{Flow in a single pipe with fast sinusoidal variation (3 cycles over 1 hour) in outlet flow with maximum magnitudes of 120, 300, 400, and 600 kg/s. From top to bottom: Outlet flow boundary condition; simulated outlet pressure (with inertia term); inlet flow (with inertia term); outlet pressure (inertia term omitted); inlet flow (inertia term omitted). The monotonicity theorem is not practical for the fast transient regime.} \label{fig:onepipefast} 
} 
\end{figure}

\begin{figure}[!t]
\centering{
\includegraphics[width=\linewidth]{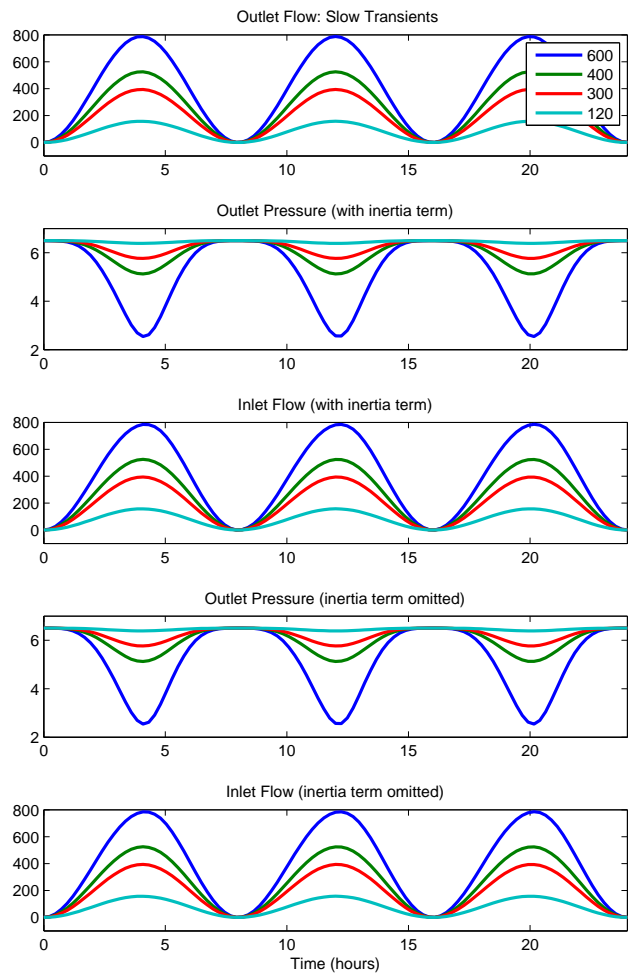} \caption{Flow in a single pipe with slow sinusoidal variation (3 cycles over 24 hours) in outlet flow with maximum magnitudes of 120, 300, 400, and 600 kg/s. From top to bottom: Outlet flow boundary condition; simulated outlet pressure (with inertia term); inlet flow (with inertia term); outlet pressure (inertia term omitted); inlet flow (inertia term omitted). Simulations including and omitting the inertia term $\partial \phi/\partial t$ are indistinguishable, so the monotonicity theorem can be applied for slow transients which are typical in practice.} \label{fig:onepipeslow} 
}
\end{figure}

\anatoly{
\section{Computational demonstrations} \label{sec:example}

In this section, we examine several computational examples in order to demonstrate the main result, as expressed in Theorem \ref{thm:pde} and Theorem \ref{thm:ode}.  In order to connect the physical modeling in Section \ref{sec:background} to the mathematical formulation in Section \ref{sec:formulation}, we first study a single pipe example to illustrate the restrictions on the transient regime in which the presented theory is relevant.  Then, we examine perturbations to an IBVP that is synthesized based on a model of an actual working gas pipeline system and associated measurement time-series, in order to verify that the main results do hold for gas pipeline systems in the ordinary operating regime.  Additionally, several simulations involving a small test network are presented in order to demonstrate an application of the monotonicity property for real-time control and an illustration of the nodal monitoring policy in practice.

\begin{figure*}[th!]
\centering{
\includegraphics[width=.99\linewidth]{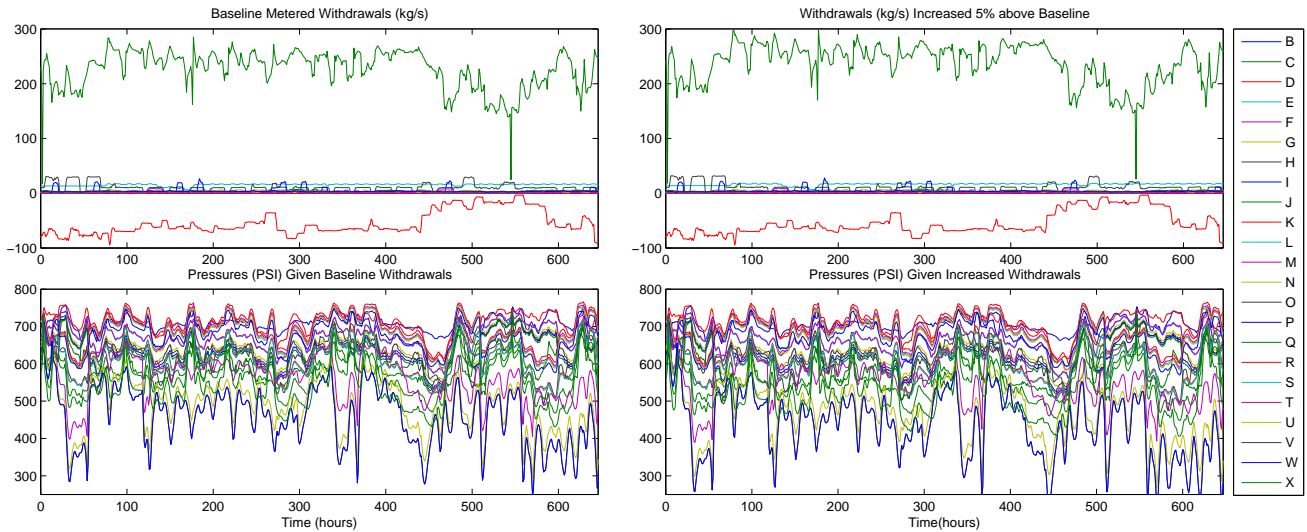} \vspace{-2ex} \caption{Testing the application of the monotone order propagation property in the normal operating regime of gas transmission pipelines. Top left: Baseline withdrawals (kg/s) custody transfer stations.  Top right: Increase of withdrawals above baseline by 5\%. Bottom left: Simulated pressure (gauge pounds per square inch (PSI)) solutions given baseline withdrawals. Note that 1 PSI equals 6894.76 Pascal.  Bottom right: Simulated pressure solutions given increased withdrawals. The letters in the legend labels correspond to the meter location tags in Figure \ref{fig:line200}.} \label{fig:line200sim} }
\vspace{-2ex}
\end{figure*}

\begin{figure}[t!]
\centering{
\includegraphics[width=.99\linewidth]{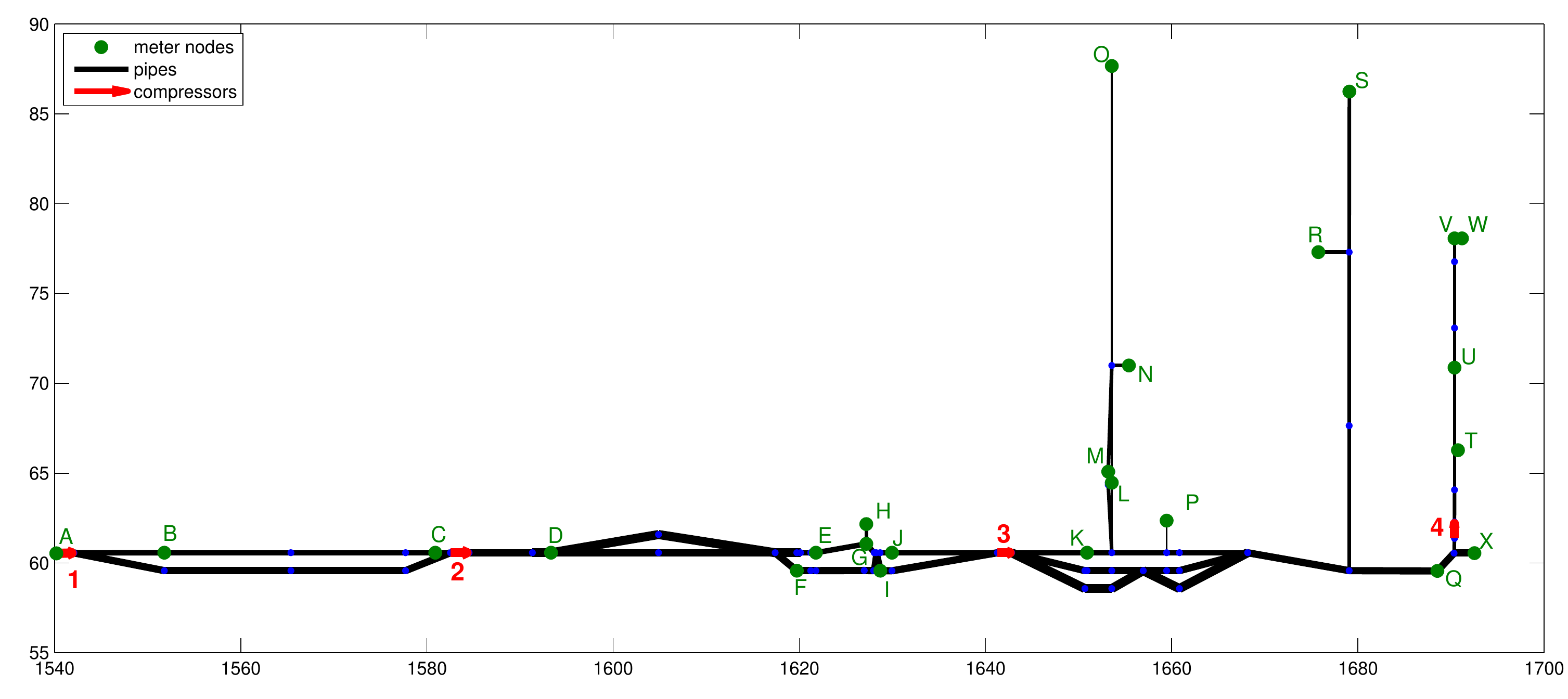} \caption{Schematic of pipeline subsystem \cite{zlotnik2017economic}. The green circles (labelled A to X) denote metered custody transfer locations. Nodes without meters are unlabelled. Red arrows (labelled 1 to 4) denote compressor stations.} \label{fig:line200} }
\vspace{-2ex}
\end{figure}

\begin{figure}[t!]
\centering{
\includegraphics[width=.99\linewidth]{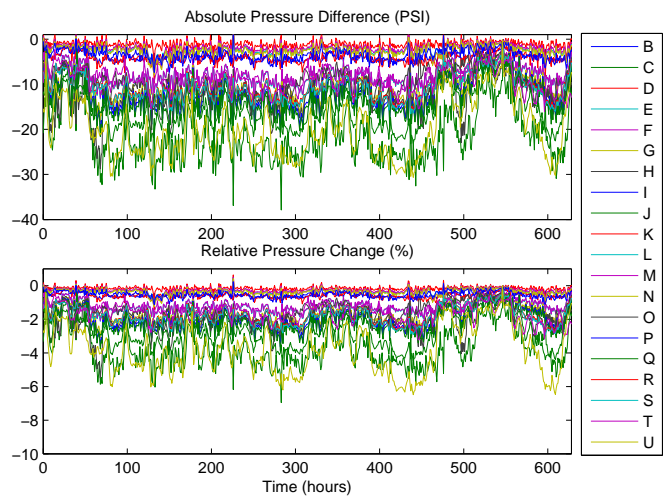}  \caption{Comparison of simulated pressure solutions at meter locations. Top: Absolute pressure difference (PSI).  Bottom: relative pressure difference (\%). With the exception of some minor deviations, the difference is negative. } \label{fig:line200comp} }
\vspace{-2ex}
\end{figure}


\subsection{Verification of the monotonicity property}
Consider a simulation of compressible gas flow on a single pipe as given by the system of equations \eqref{eq:gaspde1}, where the dependence of the gas compressibility factor $Z(p,T)$ on temperature is fixed while its dependence on pressure is given by the CNGA formula \cite{menon05,gyrya19} using a specific gas constant of $R=473.92$ J$\cdot$Kg$^{-1}\cdot(^\circ$K$)^{-1}$ and constant temperature $T=288.706$ $^\circ$K.  The pipe is 20 km in length, with a diameter of $0.9144$ meters and friction factor $\lambda=0.01$.  We fix the inlet pressure at 6.5 MPa, while sinusoidally varying the outlet flow in a series of simulations that are initialized with constant pressure in space and zero flow through the pipe.  We consider fast and slow variation in the boundary flow, and solve the IBVP for each regime including and then omitting the inertia term $\partial \phi/\partial t$ in \eqref{eq:gaspde1:b}.  All simulations are repeated for maximum magnitudes of the outlet flow that range from 120 kg/s to 600 kg/s.  The results are shown in Figures \ref{fig:onepipefast} and \ref{fig:onepipeslow} for fast and slow transients, respectively.  By the main result in Theorem \ref{thm:pde}, we expect that monotone ordering of the outlet flow should result in monotone ordering of the outlet pressure and inlet flow.  It is evident that this result does not hold for fast transients while for slow transients it does.  The key conclusion is that the monotone order propagation property does not hold in the physical regime of transient flows with fast boundary changes in which the inertia of the gas affects its momentum on the same order of magnitude as the resistance cause by turbulent drag. This is expected since the assertion in the theorems do not hold in this non-standard regime.

\begin{figure*}[t!]
\centering{
\includegraphics[width=.275\linewidth]{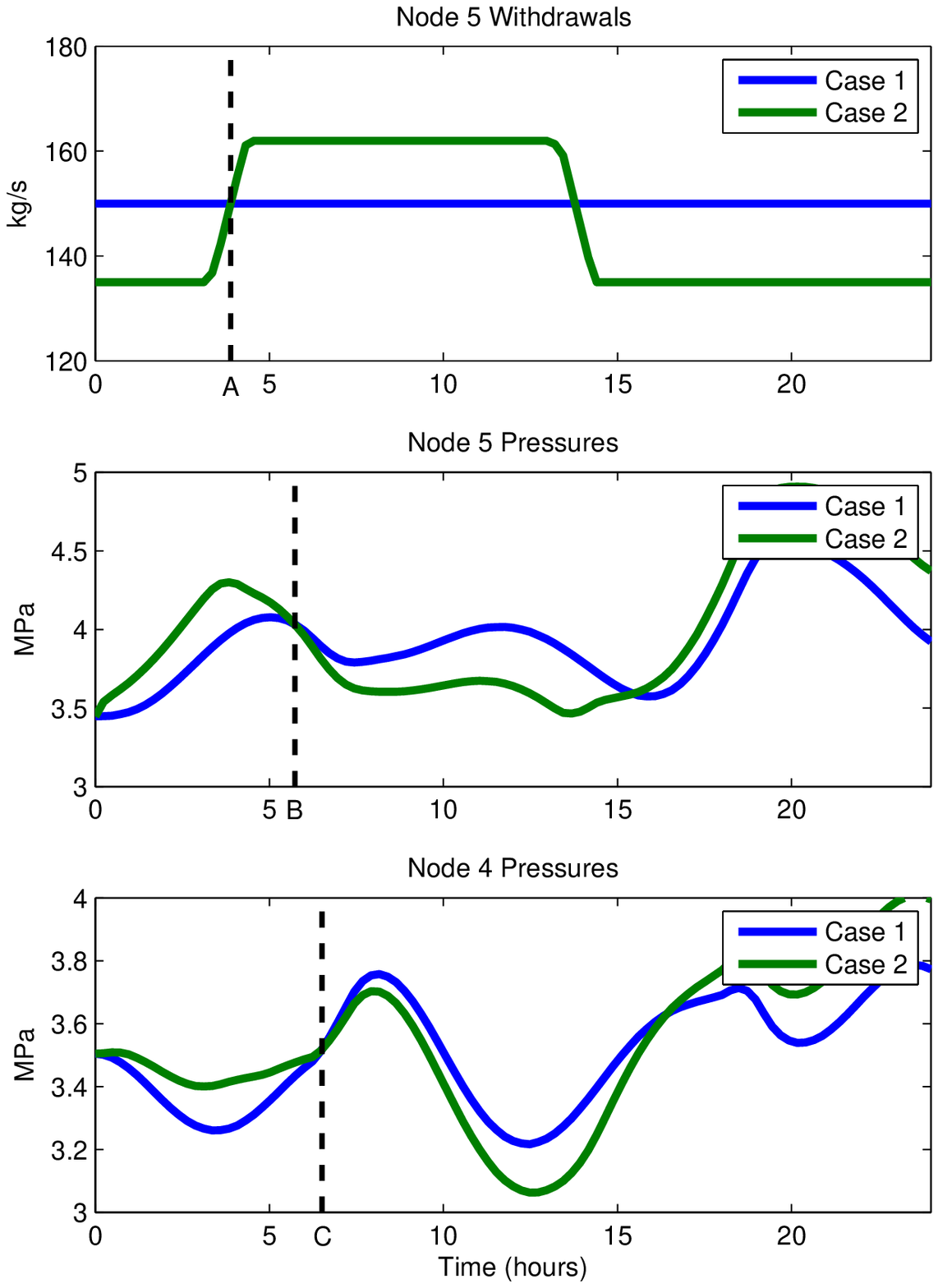} \includegraphics[width=.275\linewidth]{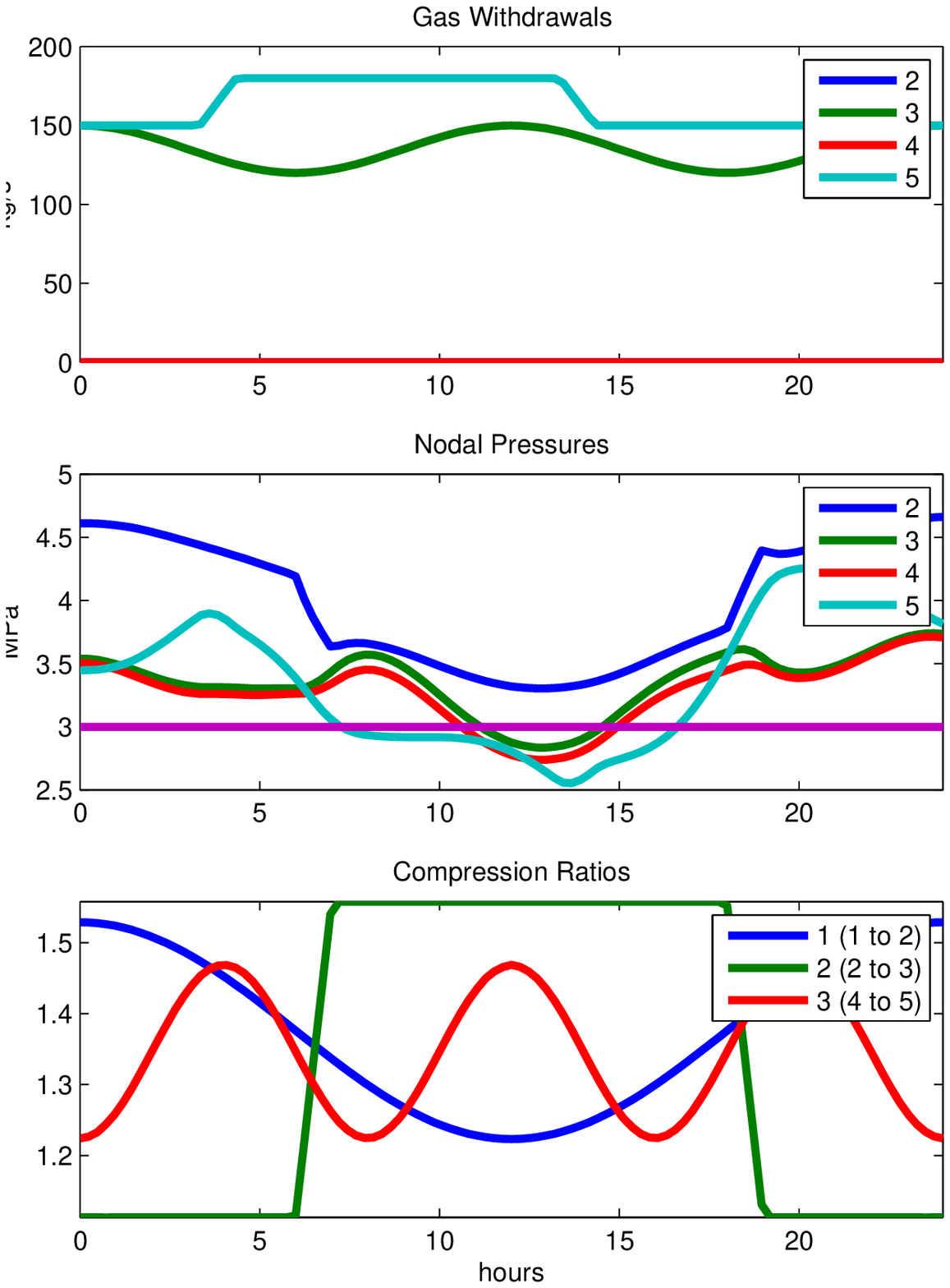} 
\includegraphics[width=.425\linewidth]{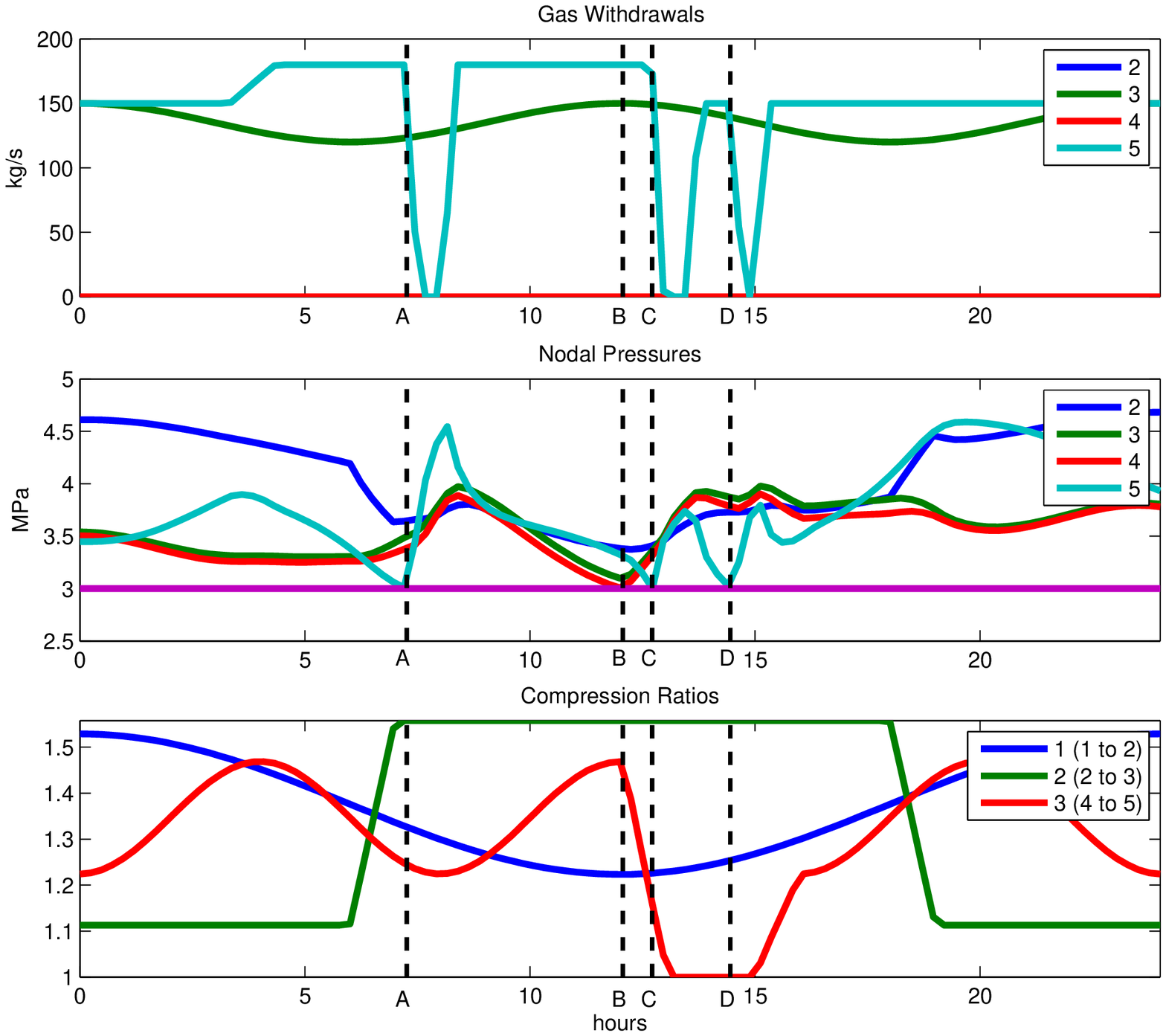} \caption{Left column: Illustration of the first crossing point property on the 5-node test network with two cases differing only by their withdrawals at node 5. Top: gas withdrawals at node 5 with order reversal between Case 1 (blue) and Case 2 (green) at $t=3.8888$ hours (A). Middle: pressures at node 5 with first crossing point at $t=5.7325$ hours (B).  Bottom: pressures at node 4 show crossing point later, at $t=6.5087$ hours (C).  Center and Right columns: Illustration of a monotonicty based real-time control policy on the 5-node test network. Center: Withdrawals at node 3 (green) and 5 (turquoise) that results in violations of the lower pressure constraint (magenta) at node 3 (green), 4 (red) and 5 (turquoise); Right: Application of the policy results in feasible pressures throughout the simulation. Control actions (black vertical line) are taken at the time when a pressure bound is reached and at the node where the pressure bound is reached. Top: gas flow withdrawals; Middle: resulting nodal pressures (minimum bound at 3 MPa is indicated in magenta); Bottom: oriented compression ratios. The policy is invoked over 15 minute ramps; (A): At $t=7.26$ hours, flow to node 5 is ramped down to zero for an hour; (B): at $t=12.06$ hours, compressor 3 is shut down; (C): at $t=12.71$ hours, flow to node 5 is curtailed for another hour; (D): at $t=14.45$ hours, flow to node 5 is curtailed again.}\label{fig:smallnetnmp}}  
\vspace{-2ex}
\end{figure*}

\begin{figure}[t!]
\centering{
\includegraphics[width=.9\linewidth]{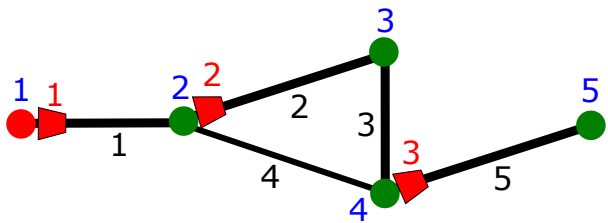} \caption{Schematic of 5-node test network \cite{gyrya19}. The green circles (labelled 2 to 5) denote flow nodes where withdrawals are given, and red circle denotes a slack node where pressure is given. Red boxes (labelled 1 to 3) denote compressors.} \label{fig:model5} }
\vspace{-2ex}
\end{figure}


Next we consider a model of an actual working pipeline subsystem with topology illustrated in Figure \ref{fig:line200}, which was simplified to basic components from a capacity planning model as part of a previous study that was used to validate optimal control modeling \cite{zlotnik2017economic}.  The network consists of 95 pipes with total length of 444.25 miles (714.95 km), which are connected at 78 reduced model nodes, and through which flow is boosted by 4 compressors. In addition, we obtain time-series data from a supervisory control and data acquisition (SCADA) system used for operation of the pipeline, which contains measurements of pressure, temperature, and metered flow leaving the system at 31 custody transfer stations and check measurement locations. Check measurements at the 4 compressor stations include pressure and temperature at suction and discharge. From this data, we use the CNGA equation of state \cite{menon05} to compute mass flow withdrawals at network nodes.  We then synthesize an IBVP where we start with steady-state initial conditions and transition to the boundary flows in the data.  For the simulation, gas density is specified at the node labeled A in Figure \ref{fig:line200}, and flows leaving the system are specified at nodes labeled B to X.  In Figure \ref{fig:line200sim}, we show the baseline flow withdrawals at B to X and resulting pressures, as well as the same simulation with all positive withdrawals increased by 5\%.  We see that, for the most part, increasing the withdrawals decreases pressures, and this is clearly displayed in Figure \ref{fig:line200comp}, which shows the absolute and relative difference between the pressure trajectories resulting given the increased and baseline outflows. The minor violations of the monotonicity property can be attributed to small contributions from the inertial term mentioned in the previous paragraph as well as numerical precision.

This simulation test provides an empirical validation of the monotone order propagation property for gas pipeline flows in the normal operating regime of relatively slowly varying transients.  The key conclusion from this analysis is that, if the pressure trajectories that result given the baseline and increased flow profiles are both considered feasible, then any uncertain flows that are bounded by these two profiles result in feasible pressures as well. With this property, solutions of optimal control problems of the form \eqref{eq:ocp1} can be certified as feasible for any such interval uncertainty.


\subsection{Nodal monitoring and real-time control}
The next computational study demonstrates the first crossing-point property, an application of the nodal monitoring policy, and an application of monotonicity for real-time control. We consider a small test network that is specified in detail in a previous study \cite{gyrya19}, and that consists of 5 nodes connected by 5 pipes and with 3 nodally located compressors as shown in Figure \ref{fig:model5}.  Here we reproduce a similar IBVP as specified in the previous study, except with the compression ratio of compressor 3 reduced to $c_3(t)=c_3(0)\cdot(1+\frac{1}{10}(1-\cos(6\pi t/T_0)))$.  Recall that the NMP takes advantage of the first crossing of trajectories at network nodes in situations when monotone ordering does not hold.  The first crossing point property is illustrated in the left column of Figure \ref{fig:smallnetnmp}, where changing ordering of withdrawals at a node for two IBVP simulations causes the first crossing of ordering in pressure for the two cases to occur at the same node. The crossing at node $4$ denoted by C is after that of node $5$ denoted by B. The crossing of all other nodes (not shown in the figure) either do not exist or are significantly later in time. This property leads to a spatial localization in monitoring requirement, where all monitoring resources can be focused on only a small subset of nodes where the injections are outside the intervals used during the planning phase.

Figure \ref{fig:smallnetnmp} (right) also shows the results of two additional IBVP simulations. These figures demonstrate an application of monotonicity property for real-time monitoring and control, as opposed to the robust planning problem outlined in \eqref{eq:ocp1}.  Suppose that 3 MPa is specified as a minimum operating pressure bound for all nodes in the system, in which case an initial (center) simulation results in pressures below the allowable minimum. Based on the monotonicity principle, whenever a pressure lower bound is approaching, either the gas withdrawal at that node is curtailed, or the compression from that node into adjacent pipes (which also serves to decrease the rate of flow leaving that node) is reduced. The resulting simulation is presented on the right, in which all nodal pressures remain above the lower bound of 3 MPa. Note that total curtailments of flow and total shutdowns of compressors are undesirable in practice as such measures can significantly increase flow volatility.
}

%

\section{Conclusions} \label{sec:conc}

In this study we have considered formulations for modeling physical flows on networks characterized by a class of parabolic partial differential equation (PDE) systems.  We derive conditions for monotonicity properties that are particularly advantageous when applied to optimize the flows of commodities over networks subject to uncertainty in parameters.  In particular we consider a class of problems where actuator control protocols and/or a subset of withdrawals are selected to optimize an economic or operational cost objective subject to nodal commodity withdrawal limits and bounds on the nodal potentials, and with interval uncertainty in another subset of flows that results in a semi-infinite problem formulation.   When the potentials are monotone functions of the withdrawals, the infinite collection of constraints that enforce the potential limits can be satisfied if the two bounds for the minimum and maximum values of the uncertainty interval are satisfied.  

Here we have derived theorems that establish the desired monotonicity properties using specialized approaches for the steady-state and transient flow regimes, as well as a monotonicity-preserving spatial discretization scheme for initial boundary value problems (IBVPs).  Specifically, we have proved for the system considered that ordering properties of solutions to the IBVP are preserved when the initial conditions and time-varying coupling law parameters at vertices are appropriately ordered.  The results have implications for robust optimization and optimal control formulations and real-time monitoring of uncertain dynamic flows on networks. Computational studies were used to demonstrate the relevance of the main results to the control of gas pipeline flows. The first key outcome is a tractable robust optimal control formulation that requires enforcing the physical flow constraints for only the extremal scenario as well as a nominal scenario.  The second implication is the ability to formulate a monitoring policy that actually allows temporary excursions of the nodal injections outside of the nominal interval uncertainty envelope.  The theory presented here, together with policies that account for system-specific engineering and operational requirements, enables tractable algorithms for robust optimization and optimal control of such systems under uncertainty.

\appendices 

\section{Proof of Theorem~\ref{thm:ss}} \label{sec:proof_ss}
The proof presented in this section is based on fundamental properties of network flows. A special case of this proof was presented in \cite{vuffray15cdc}. The proof consists of three main components. The first, which is stated in the following proposition, is a property of any network flows that obey the standard flow conservation equations in \eqref{eq:ss_flow_conservation}.
Because the flows $\phi_{ij}$ are constant along the length of each edge $(i,j)\in\cE$, i.e., independent of $x_{ij}$, it is convenient for ease of exposition to define the skew-symmetric variables given by $\phi_{ji} = -\phi_{ij}$ for all $(i,j) \in \cE$. We also define the combined neighborhood for each vertex $j \in \cV$ as $\partial j = \partial_+ j \cup \partial_- j$.

\begin{proposition}[Aquarius Theorem]
\label{thm:aquarius}
Consider two sets of flow injections $q_i^{(1)}$ and $q_i^{(2)}$ and let $\mathcal{S} \subseteq \cV$ be an arbitrary subset of $\cV$ such that for all $i \in \mathcal{S}$ we have $q_i^{(1)} \geq q_i^{(2)}$.
Let $(\phi_{ij}^{(1)})_{(i,j) \in \cE}$ and $(\phi_{ij}^{(2)})_{(i,j) \in \cE}$ be any solution to the flow conservation equations \eqref{eq:ss_flow_conservation} corresponding to the inputs $q^{(1)}$ and $q^{(2)}$. Then for every node $i \in \mathcal{S}$ there exists a non-intersecting path $i_1, \ldots, i_n$ where $i_1 \in \cV \setminus \mathcal{S}$ and $i_n = i$ such that $\phi_{i_l i_{l+1}}^{(1)} \leq \phi_{i_l i_{l+1}}^{(2)}$ for all $l = 1,\ldots, n-1$. Moreover, if $q_i^{(1)} > q_i^{(2)}$ then all the above flow inequalities are strict.
\end{proposition}

The second and third pieces of the proof contained in the following propositions are properties of the dissipation equation \eqref{eq:ss_dissipation}. 
\begin{proposition}
\label{thm:edge_monotonicity}
Fix $(i,j) \in \cE$. Consider two scenarios such that $\rho_i^{(1)} \geq \rho_i^{(2)}$ and $\phi_{ij}^{(1)} \leq \phi_{ij}^{(2)}$. Then we have $\rho_j^{(1)} \geq \rho_j^{(2)}$ and $\rho_{ij}^{(1)}(x)\geq\rho_{ij}^{(2)}(x)$ for all $x\in [0,L_{ij}]$. 
Similarly if we have two scenarios such that $\rho_j^{(1)} \geq \rho_j^{(2)}$ and $\phi_{ji}^{(1)} \leq \phi_{ji}^{(2)}$, then $\rho_i^{(1)} \geq \rho_i^{(2)}$ and $\rho_{ij}^{(1)}(x)\geq\rho_{ij}^{(2)}(x)$ for all $x\in [0,L_{ij}]$.
\end{proposition}

\begin{proposition}
\label{thm:inside_edge_monotonicity}
Fix $(i,j) \in \cE$. Consider two scenarios such that $\rho_i^{(1)} \geq \rho_i^{(2)}$ and $\rho_j^{(1)} \geq \rho_j^{(2)}$.
Then we have $\rho_{ij}^{(1)}(x) \geq \rho_{ij}^{(2)}(x)$ for all $x_{ij} \in [0,L_{ij}]$.
\end{proposition}

We prove Theorem~\ref{thm:ss} using the above three propositions. 

\begin{proof}[Proof of Theorem~\ref{thm:ss}]
We first prove that the nodal densities satisfy $\rho_i^{(1)} \geq \rho_i^{(2)}$ for all $i\in \mathcal{S}$. By the premise of the theorem, we have for all $i \in \cV \setminus \mathcal{S}$ that $\rho_i^{(1)} \geq \rho_i^{(2)}$. Fix any $j \in \mathcal{S}$. By Proposition~\ref{thm:aquarius} there exists a non-intersecting path $i_1, \ldots, i_n$ where $i_1 \in \cV \setminus \mathcal{S}$ and $i_n = j$ such that $\phi_{i_l i_{l+1}}^{(1)} \leq \phi_{i_l i_{l+1}}^{(2)}$ for all $l = 1,\ldots, n-1$. The rest is proved by induction. For the base case in the induction, since $i_1 \in \cV \setminus \mathcal{S}$ we have that $\rho_{i_1}^{(1)} \geq \rho_{i_1}^{(2)}$. We make the induction hypothesis that $\rho_{i_l}^{(1)} \geq \rho_{i_l}^{(2)}$ for $1 \leq l \leq n-1$. Now consider the edge $(i_l, i_{l+1})$. By induction hypothesis, we have $\rho_{i_l}^{(1)} \geq \rho_{i_l}^{(2)}$ and by the choice of the path we have $\phi_{i_l i_{l+1}}^{(1)} \leq \phi_{i_l i_{l+1}}^{(2)}$. Applying Proposition~\ref{thm:edge_monotonicity} for the edge $(i_l,i_{l+1})$ we get that $\rho_{i_{l+1}}^{(1)} \geq \rho_{i_{l+1}}^{(2)}$, and the induction is complete.

Now that it has been established that all nodal densities satisfy $\rho_i^{(1)} \geq \rho_i^{(2)}$, it remains to apply Proposition~\ref{thm:inside_edge_monotonicity} at every edge to end the proof.
\end{proof}

\begin{proof}[Proof of Proposition~\ref{thm:aquarius}]
We construct the required path by induction. Fix $i \in S$. Define a non-decreasing sequence of subsets $B_k \subset \cV$ and $A_k = \cup_{l=1}^{k}B_k$ as follows:
\begin{align*}
&B_1 = \{ i\} \\
&B_{k+1} = \{v \in \cV \setminus A_k \ \mid \ \exists j\in \partial v \cap B_k \mbox{ s.t. } \phi_{vj}^{(1)} \leq \phi_{vj}^{(2)}  \}.
\end{align*}
We will show that the sets defined above are strictly increasing until we encounter some vertex in $\cV \setminus S$. More precisely, we will show that there exists a positive integer $K \geq 1$ such that $B_k \neq \emptyset$ for all $1 \leq k \leq K$, and the first $K-1$ sets satisfy $B_i \cap (\cV \setminus S) =  \emptyset$ for $1\leq i \leq K-1$ and $B_K \cap (\cV \setminus S) \neq \emptyset$. We prove the above statement by induction. For the base case we have by definition $B_1 = \{i\} \neq \emptyset$. Let $\cE_k \subset \cE$ denote all the edges connecting two vertices in $A_k$, 
\begin{align}
\cE_k = \{(u,v) \in \cE \ \mid \ u,v \in A_k \}.
\end{align}
By summing the continuity equation \eqref{eq:ss_flow_conservation} over the vertices in $A_k$ we get 
\begin{align*}
    0 &= \sum_{u \in A_k} \left( q_u^{(1)} + \sum_{j \in \partial_+ u}\phi_{ju}^{(1)} + \sum_{j \in \partial_- u}\phi_{ju}^{(1)} \right) \\
    &= \sum_{u \in A_k} q_u^{(1)} + \sum_{u \in A_k}\left(   \sum_{j \in \partial_+ u}\phi_{ju}^{(1)} + \sum_{j \in \partial_- u}\phi_{ju}^{(1)} \right)  \\
    &= \sum_{u \in A_k} q_u^{(1)} + \sum_{(u,v) \in \cE_k} (\phi_{uv}^{(1)} + \phi_{vu}^{(1)}) + \sum_{u \in A_k, v \in \partial u \setminus A_k} \phi_{vu}^{(1)}\\
    &\stackrel{(a)}{=} \sum_{u \in A_k}q_u^{(1)} + \sum_{{u \in A_k},{v \in \partial u \setminus A_k} } \phi_{vu}^{(1)} ,
\end{align*}
where $(a)$ follows because the rest of the terms in the summation cancel out by the skew-symmetry property of $\phi$. Hence,
\begin{align*}
    \sum_{{u \in A_k},{v \in \partial u \setminus A_k} } \phi_{vu}^{(1)} &= -\sum_{u \in A_k} q_u^{(1)} \leq -\sum_{u \in A_k} q_u^{(2)} \\
    &= \sum_{{u \in A_k},{v \in \partial u \setminus A_k} } \phi_{vu}^{(2)} 
\end{align*}
The above inequality implies that there exists $u_k \in A_k$ and $v_k \in \partial u_k \setminus A_k$  such that $\phi_{v_k u_k}^{(1)} \leq \phi_{v_k u_k}^{(2)}$. 
Observe that $u_k \notin A_{k-1}$ because otherwise by definition of $A_k$ we must have $v_k \in B_k \subset A_k$ which contradicts the fact that $v_k \notin A_k$. Also by construction we have $u_k \in A_k$. Hence $u_k \in A_k \setminus A_{k-1} = B_k$ and $v_{k} \in B_{k+1} \neq \emptyset$. This concludes the proof by induction that $B_k \neq \emptyset$ for all $1 \leq k \leq K$. In addition, because $A_k$ is a strictly increasing sequence of sets, we must have $K \leq |\cV|$.

The construction of the required path is now straightforward. By construction, for each $1 \leq k \leq K$ and for each $ \in B_k$ there exists a path of length $k$ from $u$ to $i$. Because $B_K \cap (\cV \setminus S) \neq \emptyset$ holds, it follows that the proof is complete.
\end{proof}

\begin{proof}[Proof of Proposition~\ref{thm:edge_monotonicity}]
Consider the first case where $\rho_i^{(1)} \geq \rho_i^{(2)}$ and $\phi_{ij}^{(1)}\leq \phi_{ij}^{(2)}$. Because $\rho_i^{(1)} \geq \rho_i^{(2)}$ and $\underline{\alpha}_{ij}(.)$ is an increasing function, we have $\rho_{ij}^{(1)}(0) \geq \rho_{ij}^{(2)}(0)$. We prove the proposition by contradiction. Define the first spatial crossing point as 
\begin{align}
x_{ij}^c \triangleq \sup \{x_{ij} \in I_{ij} \, : \, \rho_{ij}^{(1)}(\hat{x}_{ij}) \geq \rho_{ij}^{(2)}(\hat{x}_{ij}) \,\, \forall \hat{x}_{ij} \in [0,x_{ij}]  \}.
\end{align}
If $x_{ij}^c = L_{ij}$ then we have $\bar{\rho}_{ij}^{(1)} \geq \bar{\rho}_{ij}^{(2)}$ and since $\bar{\alpha}_{ij}(.)$ is increasing we have $\rho_j^{(1)} \geq \rho_j^{(2)}$ and the claim in the proposition holds. For the sake of contradiction, assume that $x_{ij}^c < L_{ij}$. Then we must have 
\begin{align}
\rho_{ij}^{(1)}(x_{ij}^c) = \rho_{ij}^{(2)}(x_{ij}^c) \label{eq:ss_crossing_equality}
\end{align}
and there exists $\delta > 0$ such that 
\begin{align}
\rho_{ij}^{(1)}(x_{ij}) < \rho_{ij}^{(2)}(x_{ij}) \quad \forall x_{ij} \in (x_{ij}^c,x_{ij}^c+\delta). \label{eq:ss_crossing}
\end{align}
From the above two equations, we must have 
\begin{align}
    \pp_x \rho_{ij}^{(1)}(x_{ij}^c) \leq \pp_x \rho_{ij}^{(2)}(x_{ij}^c). \label{eq:ss_first_derivative_ordering}
\end{align}
We can classify the condition in \eqref{eq:ss_first_derivative_ordering} into two cases: 
\begin{align}
\mbox{Case 1: } \quad &\pp_x \rho_{ij}^{(1)}(x_{ij}^c) = \pp_x \rho_{ij}^{(2)}(x_{ij}^c), \label{eq:ss_case1} \\
\mbox{Case 2: } \quad &\pp_x \rho_{ij}^{(1)}(x_{ij}^c) < \pp_x \rho_{ij}^{(2)}(x_{ij}^c). \label{eq:ss_case2}
\end{align}

Observe that if Case 1 holds, then we have $\phi_{ij}^{(1)} = \phi_{ij}^{(2)} = \phi_{ij}$ and hence $\rho_{ij}^{(1)}(x_{ij})$ and $\rho_{ij}^{(2)}(x_{ij})$ are solutions to the initial value problem given by 
\begin{align}
    f_{ij}(\rho_{ij}(x_{ij}), \pp_x \rho_{ij}(x_{ij})) + \phi_{ij} = 0, \quad x_{ij} \in (x_{ij}^c,x_{ij}^c+\delta),
\end{align}
with the initial value given by $\rho_{ij}(x_{ij}^c) = \rho_{ij}^{(1)}(x_{ij}^c) = \rho_{ij}^{(2)}(x_{ij}^c)$. By the uniqueness property in Assumption~\ref{thm:ss_assumption}, we must have 
\begin{align}
\rho_{ij}^{(1)}(x_{ij}) = \rho_{ij}^{(2)}(x_{ij}) \quad \forall x_{ij} \in (x_{ij}^c,x_{ij}^c+\delta),
\end{align}
which is in direct contradiction with \eqref{eq:ss_crossing}. This eliminates Case 1. Assuming that Case 2 holds, we have from \eqref{eq:ss_crossing_equality} and \eqref{eq:ss_case2} that 
\begin{align}
\phi_{ij}^{(1)} =& -f_{ij}(\rho_{ij}^{(1)}(x_{ij}^c), \pp_x \rho_{ij}^{(1)}(x_{ij}^c)) \nonumber \\
    \stackrel{(a)}{>}& -f_{ij}(\rho_{ij}^{(2)}(x_{ij}^c), \pp_x \rho_{ij}^{(2)}(x_{ij}^c)) = \phi_{ij}^{(2)}, \label{eq:ss_contradiction}
\end{align}
where $(a)$ holds because $f_{ij}$ is increasing in its second argument. The conclusion in \eqref{eq:ss_contradiction} violates the premises of the proposition and thus our proof by contradiction is complete.

The second part of the proposition where $\rho_j^{(1)} \geq \rho_j^{(2)}$ and $\phi_{ji}^{(1)} \leq \phi_{ji}^{(2)}$ is proved using similar arguments.
\end{proof}

\begin{proof}[Proof of Proposition~\ref{thm:inside_edge_monotonicity}]
Since $\underline{\alpha}_{ij}(.)$ and $\overline{\alpha}_{ij}(.)$ are increasing functions, we have that $\rho_{ij}^{(1)}(0) \geq \rho_{ij}^{(2)}(0)$ and $\rho_{ij}^{(1)}(L_{ij}) \geq \rho_{ij}^{(2)}(L_{ij})$. For the sake of contradiction, if the proposition is not true then the two nodal densities must cross at at least two points 
$x_{c_1} \in [0,L_{ij}]$ and  $x_{c_2} \in [0,L_{ij}]$. The first crossing point $x_{c_1}$ must satisfy 
$\rho_{ij}^{(1)}(x_{c_1}) = \rho_{ij}^{(2)}(x_{c_1})$ and $\partial_x\rho_{ij}^{(1)}(x_{c_1}) \leq \partial_x\rho_{ij}^{(2)}(x_{c_1})$. Therefore,
\begin{align}
    \phi_{ij}^{(1)} &= - f_{ij}(\rho_{ij}^{(1)}(x_{c_1}) ,\partial_x\rho_{ij}^{(1)}(x_{c_1})) \nonumber\\
    &\stackrel{(a)}{\geq} - f_{ij}(\rho_{ij}^{(2)}(x_{c_1}) ,\partial_x\rho_{ij}^{(2)}(x_{c_1})) = \phi_{ij}^{(2)},    \label{eq:prop6proof1}
\end{align}
where (a) holds because $f_{ij}$ is increasing in its second argument.
At the same time, the second crossing point $x_{c_2}$ must satisfy exactly the opposite relations, i.e.,
$\rho_{ij}^{(1)}(x_{c_2}) = \rho_{ij}^{(2)}(x_{c_2})$ and $\partial_x\rho_{ij}^{(1)}(x_{c_2}) \geq \partial_x\rho_{ij}^{(2)}(x_{c_2})$, resulting in
\begin{align}
     \phi_{ij}^{(1)} &= - f_{ij}(\rho_{ij}^{(1)}(x_{c_2}) ,\partial_x\rho_{ij}^{(1)}(x_{c_2})) \nonumber\\
    &{\leq} - f_{ij}(\rho_{ij}^{(2)}(x_{c_2}) ,\partial_x\rho_{ij}^{(2)}(x_{c_2})) = \phi_{ij}^{(2)}. \label{eq:prop6proof2}
\end{align}
Combining \eqref{eq:prop6proof1} and \eqref{eq:prop6proof2}, we get $\phi_{ij}^{(1)} = \phi_{ij}^{(2)}$. However, this means that by the uniqueness property in Assumption~\ref{thm:ss_assumption}, we must have $\rho_{ij}^{(1)}(x) = \rho_{ij}^{(2)}(x)$ for all $x \in (x_{c_1},L]$. Since $x_{c_1}$ is the first crossing point, we also have that $\rho_{ij}^{(1)}(x) \geq \rho_{ij}^{(2)}(x)$ for all $x \in [0,x_{c_1}]$, thus showing that there is no crossing of the densities. 


\end{proof}

\section{Proof of Theorem~\ref{thm:pde}}  \label{sec:proof_pde}
In this section we present a direct proof of monotone ordering of solutions to IBVPs involving the equations \eqref{eq:in_continuity}-\eqref{eq:in_pressure_comp} and initial conditions \eqref{eq:in_initial_condition}.
This result is a generalization of the result previously proposed in \cite{misra16mtns}. The approach taken here is based on the notion of crossing points for solutions of the dynamic equations, at which certain variable values change ordering.  

\subsection{Crossing Points}
The proof is constructed by establishing the non-existence of the so-called ``first crossing point". We formalize this definition below.
\begin{definition} \label{def:first_crossing}
 Let $\rho_{ij}^{(1)}(t,x_{ij})$ and $\rho_{ij}^{(2)}(t,x_{ij})$ be the unique classical solutions corresponding to the initial conditions $\rho_{ij}^{(1)}(0,x_{ij})$ and  $\rho_{ij}^{(2)}(0,x_{ij})$ and injections $q^{(1)}_i(t)$ and $q^{(2)}_i(t)$
 respectively. Further suppose that for all $(i,j) \in \cE$ and for all $x_{ij} \in I_{ij}$ we have $\rho_{ij}^{(1)}(0,x_{ij}) \geq \rho_{ij}^{(2)}(0,x_{ij})$. Then a tuple $(t_c, x_c)$, where $t_c \in (0,T]$ and $x_c \in I_{ij}$ for some $(i,j) \in \cE$ is called a first crossing point if
 \begin{align}
 t_c = \sup \{t \in &[0,T] \ : \ \rho_{ij}^{(1)}(t,x_{ij}) \geq \rho_{ij}^{(2)}(t,x_{ij}) \nonumber  \\
 & \fA (i,j) \in \cE, \, x_{ij} \in I_{ij}  \}, \label{eq:crossing_time}
 \end{align}
 and, there exists a $\delta > 0$, such that
 \begin{align}
 \rho_{ij}^{(1)}(t,x_{c}) < \rho_{ij}^{(2)}(t,x_{c}),  \label{eq:positive_derivative}
 \end{align}
 for all $t \in (t_c, t_c + \delta)$.
\end{definition}

First crossing points need not be unique because there may be multiple coordinates $x_c$ that satisfy the above definition. The crossing time $t_c$ however, is unique by definition. Note that whenever there is no crossing point in the system dynamics until some time $t_0$, then the ordering of the initial conditions must be preserved till $t_0$. In the rest of the section, we prove the appropriate non-existence of first crossing points in order to establish Theorem~\ref{thm:pde}.
As a remark to the reader, a weaker notion of crossing point that only considers the spatial coordinate was utilized in the proof of Proposition~\ref{thm:edge_monotonicity} in Section~\ref{sec:steady_state}.

\subsection{Technical Lemmas}
In this section, we prove several technical lemmas that will be useful to establish Theorem~\ref{thm:pde}. Let $\rho_{ij, \epsilon}^{(1)}(t,x_{ij})$ be the solution to the perturbed system in  \eqref{perturbed1}-\eqref{perturbed2}
with nodal input parameters set at $q_{i}^{(1)}(t)$. The following lemmas prove that there can be no crossing point of the perturbed solution $\rho_{ij, \epsilon}^{(1)}(t,x_{ij})$ and $\rho_{ij}^{(2)}(0,x_{ij})$ either in the interior of an edge or at a vertex $i \in \cV$ where $q_i^{(1)}(t) \geq q_i^{(2)}(t)$.

\begin{lemma} \label{lem:perturbed_interior}
Let $(i,j) \in \cE$ be an edge. Suppose that for all $x_{ij} \in I_{ij}$ we have $\rho_{ij}^{(1)}(0,x_{ij}) \geq \rho_{ij}^{(2)}(0,x_{ij})$.
Then there is no first crossing point  $(t_c, x_c)$ between the perturbed solution $\rho_{ij, \epsilon}^{(1)}(t,x_{ij})$ and  $\rho_{ij}^{(2)}(t,x_{ij})$ such that $t_c \in [0,T]$ and $0 < x_c < L_{ij}$.
\end{lemma}

\begin{lemma} \label{lem:nodal_cross_equiv}
Let $j \in \cV$.  Suppose that a first crossing occurs at $(t_c, x_c=0)$ on an edge $(j,k)\in\cE$ for one $k\in\partial_- j$ or at $(t_c, x_c=L_{ij})$ on an edge $(i,j)\in\cE$ for one $i\in\partial_+ j$.  Then a first crossing point occurs at $(t_c, x_c=0)$ for all edges $(j,k)\in\cE$ with $k\in\partial_- j$ and also at $(t_c, x_c=L_{ij})$ for all edges $(i,j)\in\cE$ with $i\in\partial_+ j$.
\end{lemma}

\begin{lemma} \label{lem:perturbed_vertex}
Let $j \in \cV$. Suppose that for all $k \in \partial_- j$ and for all $x_{jk} \in I_{jk}$ we have $\rho_{jk}^{(1)}(0,x_{jk}) \geq \rho_{jk}^{(2)}(0,x_{jk})$, and that for all $i \in \partial_+j$ and for all $x_{ij} \in I_{ij}$ we have $\rho_{ij}^{(1)}(0,x_{ij}) \geq 
\rho_{ij}^{(2)}(0,x_{ij})$. Further, suppose that $q_j^{(1)}(t) \geq q_j^{(2)}(t)$.
Then there is no first crossing point $(t_c, x_c)$ between the perturbed solution $\rho_{jk, \epsilon}^{(1)}(t,x_{jk})$ and $\rho_{jk}^{(2)}(t,x_{jk})$ such that $t_c \in [0,T]$ and $x_c = 0$ for any $k\in\partial_-j$ or $x_c=L_{ij}$ for any $i \in \partial_+j$.
\end{lemma}

\begin{proof}[Proof of Lemma~\ref{lem:perturbed_interior}]
Suppose for the sake of contradiction that there exists a first crossing point $(t_c, x_c)$ such that $0  < x_c < L_{ij}$. Then by Definition \ref{def:first_crossing},
\begin{align}
\rho_{ij, \epsilon}^{(1)}(t_c, x_c) &= \rho_{ij}^{(2)}(t_c, x_c), \label{cp_1} \\
\rho_{ij, \epsilon}^{(1)}(t_c, x) &\geq \rho_{ij}^{(2)}(t_c, x), \ x \in (0,L_{ij}). \label{cp_2}
\end{align}
By Assumption \ref{assumption}, the functions $\rho_{ij, \epsilon}^{(1)}(t_c, x)$ and $\rho_{ij}^{(2)}(t_c, x)$ and hence the function $g: (0,L_{ij}) \rightarrow \mathbb{R}$ given by
\begin{align} \label{eq:g_fun}
g(x) &= \rho_{ij, \epsilon}^{(1)}(t_c, x)  - \rho_{ij}^{(2)}(t_c, x)
\end{align}
is twice continuously differentiable. Combined with \eqref{cp_1}-\eqref{cp_2}, this means the function $g(.)$ must also satisfy
\begin{align}
\ppx{x}g(x_c) = 0, \quad \frac{\partial}{\partial x^2}g(x_c) \geq 0, \label{cp_3}
\end{align}
which in turn yields
\begin{align}
\partial_x \rho_{ij, \epsilon}^{(1)}(t_c, x_c) &= \partial_x\rho_{ij}^{(2)}(t_c, x_c), \label{eq:cp_equality} \\
 \partial_x^2 \rho_{ij, \epsilon}^{(1)}(t_c, x_c) &\geq \partial_x^2 \rho_{ij}^{(2)}(t_c, x_c). \label{eq:cp_ineq1}
\end{align}
See Figure~\ref{fig:crossing} for a pictorial interpretation of the relations \eqref{eq:cp_equality} and \eqref{eq:cp_ineq1}.

\begin{figure}[t]
\centering
\includegraphics[width=1.0\linewidth]{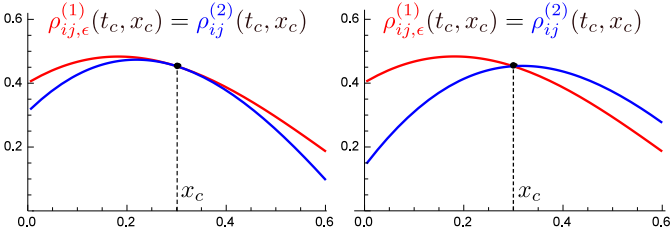} \vspace{-2ex} \caption{Left panel: example of a first crossing point at $x_c$. Right panel: example of a crossing point at $x_c$ that is not a first crossing point.  Relations \eqref{eq:cp_equality} and \eqref{eq:cp_ineq1} are satisfied in the left panel but not in the right panel.
} \label{fig:crossing} \vspace{2ex}
\end{figure}

Further, \eqref{eq:positive_derivative} implies that $\partial_t \rho_{ij, \epsilon}^{(1)}(t_c, x_c) \leq \partial_t \rho_{ij}^{(2)}(t_c, x_c)$, so that applying the continuity equation \eqref{eq:in_continuity} and its perturbed version \eqref{perturbed1}, we obtain
\begin{align}
-  \pp_x\phi_{ij,\epsilon}^{(1)}(t_c,x_c)  + \epsilon \leq - \pp_x\phi_{ij}^{(2)}(t_c,x_c). \label{eq:pert_comp}
\end{align}
We substitute for the flow terms in \eqref{eq:pert_comp} using the dissipation equation \eqref{eq:in_dissipation_eq} and its perturbed counterpart \eqref{perturbed2} to obtain the relation
\begin{align}
 \partial_x &f_{ij}(t_c,\rho_{ij, \epsilon}^{(1)}(t,x_{c}), \partial_{x}\rho_{ij, \epsilon}^{(1)}(t,x_{c})) \nonumber \\
   &\leq \partial_x  f_{ij}(t,\rho_{ij}^{(2)}(t,x_{c}), \partial_{x}\rho_{ij}^{(2)}(t,x_{c})) - \epsilon \nonumber \\
   &<  \partial_x  f_{ij}(t,\rho_{ij}^{(2)}(t,x_{c}), \partial_{x}\rho_{ij}^{(2)}(t,x_{c})). \label{eq:cp_ineq2}
\end{align}
Using the chain rule for differentiation, we can rewrite for $k = 1,2$,
\begin{align}
\!\!\! \partial_x &f_{ij}(t_c,\rho_{ij}^{(k)}(t,x_{c}), \partial_{x}\rho_{ij}^{(k)}(t,x_{c}))  \nonumber \\
\!\!\! = & \,\partial_u f_{ij}(t_c,\rho_{ij}^{(k)}(t,x_{c}), \partial_{x}\rho_{ij}^{(k)}(t,x_{c})) \partial_x \rho_{ij}^{(k)}(t,x_{c}) \nonumber\\
\!\!\!  &+  \partial_v f_{ij}(t_c,\rho_{ij}^{(k)}(t,x_{c}), \partial_{x}\rho_{ij}^{(k)}(t,x_{c})) \partial_x^2 \rho_{ij}^{(k)}(t,x_{c}). \label{eq:fx_ineq0}
\end{align}
Substituting \eqref{eq:fx_ineq0} into \eqref{eq:cp_ineq2} and using \eqref{cp_1} and \eqref{eq:cp_equality}, we get
\begin{align}
\!\!\!\!\!\!\!& \partial_v f_{ij}(t_c,\rho_{ij, \epsilon}^{(1)}(t,x_{c}), \partial_{x}\rho_{ij, \epsilon}^{(1)}(t,x_{c})) \partial_x^2 \rho_{ij, \epsilon}^{(1)}(t,x_{c})  \nonumber\\
   & \quad < \partial_v f_{ij}(t_c,\rho_{ij}^{(2)}(t,x_{c}), \partial_{x}\rho_{ij}^{(2)}(t,x_{c})) \partial_x^2 \rho_{ij}^{(2)}(t,x_{c}). \label{eq:fv_comp1}
\end{align}
Because the dissipation function $f_{ij}(t,u,v)$ is strictly increasing in the third argument $v$, and recalling the equivalence relations \eqref{cp_1} and \eqref{eq:cp_equality}, we have
\begin{align}
\!\!\!\! &\partial_v f_{ij}(t_c,\rho_{ij, \epsilon}^{(1)}(t,x_{c}), \partial_{x}\rho_{ij, \epsilon}^{(1)}(t,x_{c})) \nonumber \\
 & \quad =  \partial_v f_{ij}(t_c,\rho_{c}^{(2)}(t,x_{c}), \partial_{x}\rho_{ij}^{(2)}(t,x_{c})) > 0. \label{eq:fv_comp2}
\end{align}
Finally, the equality and positivity of $\partial_v f_{ij}$ terms in \eqref{eq:fv_comp2} can be used to simplify \eqref{eq:fv_comp1} to yield the simple strict inequality
\begin{align}
\partial_x^2 \rho_{ij}^{(1)}(t,x_{ij}) < \partial_x^2 \rho_{ij}^{(2)}(t,x_{ij}).
\end{align}
This contradicts \eqref{eq:cp_ineq1}, and hence our assumption must be incorrect and the proof of the lemma is complete.
\end{proof}


\begin{proof}[Proof of Lemma~\ref{lem:nodal_cross_equiv}]
Recall that by the compatibility constraints \eqref{eq:in_pressure_comp}, we have for any $k \in \partial_- j$ that $\rho_{jk}(t_c, 0) = \underline{\alpha}_{jk} (t_c,\rho_{j}(t_c))$ 
and for any $i \in \partial_+ j$ that $\rho_{ij}(t_c, L_{ij}) = \overline{\alpha}_{ij} (t_c,\rho_{j}(t_c))$.  Here $\underline{\alpha}_{jk}(t_c,\rho)$ and $\overline{\alpha}_{ij}(t_c,\rho)$ are invertible functions of $\rho$ for all $(i,j),(j,k)\in\cE$ and $\rho>0$, because we have assumed that the compatibility functions are strictly increasing for positive values in the second argument.  Let us then denote by $\underline{\alpha}_{jk,t}^{-1}(\cdot)$ and $\overline{\alpha}_{ij,t}^{-1}(\cdot)$ the inverses of the corresponding functions at the time $t$.  Then, for any $i\in\partial_+ j$ and $k\in\partial_- j$, we have the relations
\begin{align}
\rho_{ij}(t_c, L_{ij})& = \overline{\alpha}_{ij} (t_c,\underline{\alpha}_{jk,t_c}^{-1}(\rho_{jk}(t_c, 0))), \label{eq:bij1}\\
\rho_{jk}(t_c, 0) &= \underline{\alpha}_{jk} (t_c,\overline{\alpha}_{ij,t_c}^{-1}(\rho_{ij}(t_c, L_{ij}))), \label{eq:bij2}
\end{align}
where $\overline{\alpha}_{ij} (t,\underline{\alpha}_{jk,t}^{-1}(\cdot))$ and $\underline{\alpha}_{jk} (t,\overline{\alpha}_{ij,t}^{-1}(\cdot))$ are compositions of invertible increasing functions and therefore bijective and increasing.  As a result, by Definition \ref{def:first_crossing}, the existence of a crossing  point at $(t_c,x_{jk}=0)$ for some $k\in\partial_- j$ implies that there is also a crossing point at $(t_c, x_{jk} = 0)$ for all $k \in \partial_- j$ and also at $(t_c, x_{ij} = L_{ij})$ for all $i \in \partial_+ j$.
\end{proof}

\begin{proof}[Proof of Lemma~\ref{lem:perturbed_vertex}]
We again seek to reach a contradiction by starting with the assumption that there exists a first crossing point $(t_c,x_c)$ for some $t_c \in [0,T]$ between the quantities  $\rho_{jk, \epsilon}^{(1)}(t, x_{jk})$ and
 $\rho_{jk}^{(2)}(t, x_{jk})$ for some  $k \in \partial_- j$ and $x_{jk}=x_c = 0$. By the definition of first crossing point we must have
\begin{align}
\rho_{jk, \epsilon}^{(1)}(t_c, 0) &= \rho_{jk}^{(2)}(t_c, 0), \label{cpv_1} \\
\rho_{jk, \epsilon}^{(1)}(t_c, x) &\geq \rho_{jk}^{(2)}(t_c, x), \ x \in [0,L_{jk}]. \label{cpv_2}
\end{align}
We then can apply a similar argument as that used in the proof of Lemma \ref{lem:perturbed_interior}.  Because $\rho_{jk}^{(1)}(t_c, x)  - \rho_{jk}^{(2)}(t_c, x)$ is twice continuously differentiable,
one of the following options must be true.
\begin{itemize}
\item \emph{Option 1:}
\begin{align}
\partial_x \rho_{jk, \epsilon}^{(1)}(t_c, 0) &= \partial_x\rho_{jk}^{(2)}(t_c, 0),  \\
 \partial_x^2 \rho_{jk, \epsilon}^{(1)}(t_c, 0) &\geq \partial_x^2 \rho_{jk}^{(2)}(t_c, 0).
\end{align}
\item \emph{Option 2:}
\begin{align}
\partial_x \rho_{jk, \epsilon}^{(1)}(t_c, 0) &> \partial_x\rho_{jk}^{(2)}(t_c, 0). \label{eq:cpv_inequality}
\end{align}
\end{itemize}

By following the exact same argument as in the proof of Lemma~\ref{lem:perturbed_interior},  we can show that \emph{Option 1} leads to a contradiction. What remains is to prove that \emph{Option 2} is also disallowed.  Because $f_{jk}(t,u,v)$ is strictly increasing in $v$, using \eqref{cpv_1} and \eqref{eq:cpv_inequality} we see that
\begin{align}
f_{jk}(t_c, \rho_{jk}^{(1)}(t_c,0), &\partial_x \rho_{jk}^{(1)}(t_c,0))  \nonumber\\
& > f_{jk}(t_c, \rho_{jk}^{(2)}(t_c,0), \partial_x \rho_{jk}^{(2)}(t_c,0)). \label{eq:pert_ineq1}
\end{align}
Applying Lemma \ref{lem:nodal_cross_equiv} to edges outgoing from node $j$ we find that \eqref{cpv_1} and \eqref{eq:cpv_inequality} hold for all $k \in \partial_- j$, hence so does \eqref{eq:pert_ineq1}.  Combining with the dissipation equation \eqref{eq:in_dissipation_eq}, this gives for all $k \in \partial_- j$  that $\phi_{jk}^{(1)}(t_c, 0) < \phi_{jk}^{(2)}(t_c, 0)$.  Similarly, Lemma \ref{lem:nodal_cross_equiv} implies that the relations $\rho_{ij, \epsilon}^{(1)}(t_c, L_{ij}) = \rho_{ij}^{(2)}(t_c, L_{ij})$ and $\partial_x \rho_{ij, \epsilon}^{(1)}(t_c, L_{ij}) \leq \partial_x\rho_{ij}^{(2)}(t_c, L_{ij})$ must hold for all $i\in\partial_+ j$, and hence $\phi_{ij}^{(1)}(t_c, L_{ij}) \geq \phi_{ij}^{(2)}(t_c, L_{ij})$ hold for all $i \in \partial_+ j$ as well.  We then apply the flow conservation equation \eqref{eq:in_nodal_continuity} to obtain
\begin{align}
q^{(1)}(t_c) &= \sum_{k \in \partial_- j} \phi_{jk}^{(1)}(t_c, 0) -  \sum_{i \in \partial_+ j} \phi_{ij}^{(1)}(t_c, L_{ij}) \nonumber \\
& <   \sum_{k \in \partial_- j} \phi_{jk}^{(2)}(t_c, 0) -  \sum_{i \in \partial_+ j} \phi_{ij}^{(2)}(t_c, L_{ij})  \nonumber \\
&= q^{(2)}(t_c).
\end{align}
The last statement is in contradiction with the assumptions of Lemma~\ref{lem:perturbed_vertex}.
\end{proof}

Before proceeding to the proof of Theorem~\ref{thm:pde} we state one last lemma that relates the solution of the perturbed system in  \eqref{perturbed1}-\eqref{perturbed2}  to the original system.
\begin{lemma} \label{rem:perturbed_ordering}
The solution to the perturbed system $\rho_{ij,\epsilon}^{(1)}(t,x_{ij})$ is always greater than or equal to the solution $\rho_{ij}^{(1)}(t,x_{ij})$ of the original system for all $t \in [0,T]$.
\end{lemma}
\begin{proof}
We observe that all assumptions in Lemma~\ref{lem:perturbed_interior} and Lemma~\ref{lem:perturbed_vertex} are satisfied if we replace $\rho_{ij}^{(2)}(t,x_{ij})$ by  $\rho_{ij}^{(1)}(t,x_{ij})$.
As a consequence, there can be no first crossing point between $\rho_{ij,\epsilon}^{(1)}(t,x_{ij})$ and $\rho_{ij}^{(1)}(t,x_{ij})$. Therefore, for all $t \in [0,T]$ we must have $\rho_{ij,\epsilon}^{(1)}(t,x_{ij}) \geq \rho_{ij}^{(1)}(t,x_{ij})$.
\end{proof}
\vspace{0.1in}

\subsection{Proof of Theorem~\ref{thm:pde}}
\begin{proof}
Fix an $\epsilon > 0$.  Let $\rho_{ij, \epsilon}^{(1)}(t,x_{ij})$ be the solution to the perturbed system in  \eqref{perturbed1}-\eqref{perturbed2}
with nodal input parameters set at $q_{i}^{(1)}(t)$. Then by Lemma~\ref{lem:perturbed_interior}, there can be no first crossing point between $\rho_{ij, \epsilon}^{(1)}(t,x_{ij})$ and $\rho_{ij}^{(2)}(t,x_{ij})$ such that $t_c \in [0,t_0]$ and $0 < x_{ij} < L_{ij}$
for some $(i,j) \in \cE$. The above statement is also true by Lemma~\ref{lem:perturbed_vertex} for $i \in \mathcal{S}$.
Further, by Lemma~\ref{rem:perturbed_ordering}, we have that $\rho_{i}^{(1)}(t) \geq \rho_{i}^{(2)}(t)$ implies $\rho_{i, \epsilon}^{(1)}(t) \geq \rho_{i}^{(2)}(t)$ and hence there is no crossing point at $i \notin \mathcal{S}$. This means that, for all $t \in [0,t_0]$, we have
\begin{align}
\rho_{ij, \epsilon}^{(1)}(t,x_{ij}) \geq \rho_{ij}^{(2)}(t,x_{ij}), \label{eq:main_ineq_a}
\end{align}
for all $(i,j) \in \cE$ and $x_{ij} \in I_{ij}$. Because $\epsilon > 0$ was chosen arbitrarily, we can take the limit $\epsilon \rightarrow 0$ in \eqref{eq:main_ineq_a}, and using Assumption~\ref{assumption}-(iv), we get
\begin{align}
\rho_{ij}^{(1)}(t,x_{ij}) \geq \rho_{ij}^{(2)}(t,x_{ij}), \label{eq:main_ineq}
\end{align}
This completes the proof of the Main Theorem.
\end{proof}

\section{Proof of Theorem~\ref{thm:ode}}  \label{sec:proof_ode}
In this section we present a proof that the monotone order propagation property is preserved for systems of the form in equations \eqref{eq:in_continuity}-\eqref{eq:in_pressure_comp} when they are discretized through the procedure presented in Section \ref{sec:discrete_dynamics}.  Our proof relies on an application of the theory of monotone control systems \cite{angeli03} to the system of ODEs that describe the nodal density dynamics in \eqref{disceq3}. Denoting the vector of densities by $\mathbf{\rho} = \left(\rho_i \right)_{i \in S}$  and the vector of injections by $\mathbf{q} = (q_i)_{i\in S}$ for the nodes in $S$, we can rewrite the system of ODEs in \eqref{disceq3} as
\begin{align}
\dot{\mathbf{\rho}} = F(\mathbf{\rho},\mathbf{q}), \label{eq:ode_vector_form}
\end{align}
where $F_i(\mathbf{\rho},\mathbf{q})$ is defines as the right hand side of \eqref{disceq3}. 

\begin{definition}
A matrix $A \in \mathbbm{R}^{n \times n}$ is called \emph{non-negative} if all of its entries are non-negative
\end{definition}

\begin{definition}
A matrix $A \in \mathbbm{R}^{n \times n}$ is called \emph{Metzler} if all of its off-diagonal entries are non-negative, i.e.,, $A_{ij} \geq 0$ for all $i \neq j$.
\end{definition}

\begin{proposition} \label{thm:kamke_muller}
Suppose that the system of ODEs in \eqref{eq:ode_vector_form} is such that the matrix $\nabla_{\mathbf{q}} F$ is non-negative and $\nabla_{\mathbf{\rho}} F$ is Metzler, then the conclusion in Theorem~\ref{thm:ode} holds.
\end{proposition}
\begin{proof}
The proof follows from well-known results in the theory of monotone control systems as an application of the Kamke-M\"uller conditions \cite{kamke32,angeli03,hirsch05}.
\end{proof}

To prove Theorem~\ref{thm:ode}, what remains is to show that the conditions in Proposition~\ref{thm:kamke_muller} hold. 
For $i \neq j$ we have $\left[\nabla_{\mathbf{q}}\right]_{ij} = 0$. For the diagonal elements, we compute
\begin{align}
\left[\nabla_{\mathbf{q}} F\right]_{jj}= \frac{2}{\eP} \left(\frac{\partial}{\partial \rho}\alpha_j(t,\rho_j) \right)^{-1} \stackrel{(a)}{>} 0,
\end{align}
where $(a)$ follows by using \eqref{eq:alpha_def} to get $\alpha_j(t,\rho_j) = \sum_{i\in\partial_+j} \overline{\alpha}_{ij}(t,\rho_j) + \sum_{k\in\partial_-j} \underline{\alpha}_{jk}(t,\rho_j)$ and observing that each of the functions $\overline{\alpha}_{ij}(t,\rho_j)$ and $\underline{\alpha}_{jk}(t,\rho_j)$ is monotonically increasing in $\rho_j$. This shows that $\nabla_{\mathbf{q}} F$ is non-negative.

If $i \notin \partial_+ j \cup \partial_- j$ then $\left[\nabla_{\mathbf{\rho}}F\right]_{ij} = 0$. Let $i \in \partial_+ j$. Then
\begin{align}
\left[\nabla_{\mathbf{\rho}}F\right]_{ij} &= \frac{2}{\eP} \left(\frac{\partial}{\partial \rho}\alpha_j(t,\rho_j) \right)^{-1} \frac{1}{\eP}\frac{\partial}{\partial \rho} \underline{\alpha}_{ij}(t,\rho_i) \times \nonumber \\
& h_{\mu(ij)}\left(t,\overline{\alpha}_{ij}(t,\rho_j), \frac{1}{\eP}(\overline{\alpha}_{ij}(t,\rho_j)-\underline{\alpha}_{ij}(t,\rho_i)) \right),
\end{align}
where $h_{\mu(ij)}(t,u,v) = \frac{\partial}{\partial v} f_{\mu(ij)}(t,u,v) > 0$, and where the inequality follows from the assumption of Theorem~\ref{thm:ode}.  It follows that $\left[\nabla_{\mathbf{\rho}}F\right]_{ij} \geq 0$, and a similar computation shows that for $k \in \partial_- j$, the inequality $\left[ \nabla_{\mathbf{\rho}}F \right]_{jk} \geq 0$ holds as well. This proves that $\nabla_{\mathbf{\rho}}F$ is Metzler, and the proof of Theorem~\ref{thm:ode} follows by using Proposition~\ref{thm:kamke_muller}.

We note that the result established above holds for the specific model of natural gas networks.  In particular, a result similar to Theorem~\ref{thm:ode} has been applied to the case of ideal gas modeling in the context of optimal state estimation \cite{sundar2018state}.

\anatoly{
\section*{Acknowledgements} This study was conducted at Los Alamos National Laboratory under the auspices of the National Nuclear Security Administration of the U.S. Department of Energy under Contract No. 89233218CNA000001. Support was provided by the Advanced Research Projects Agency-Energy (ARPA-e) of the U.S. Department of Energy under Award No. DE-AR0000673, the Advanced Grid Research and Development program of the D.O.E. Office of Electricity, and the Laboratory Directed Research \& Development program at Los Alamos National Laboratory.  The authors gratefully acknowledge the assistance of the Gas Pipeline Group at Kinder Morgan corporation, who contributed the data set used in the computational analysis. 

}

\bibliographystyle{unsrt}
\bibliography{gas_master,monotonicity_master,cdc_gas_network}

\end{document}